\documentclass[11pt,reqno]{amsart}

\usepackage[T1]{fontenc}
\usepackage[utf8]{inputenc}
\usepackage{lmodern}
\usepackage{microtype}
\usepackage[margin=1.08in]{geometry}
\usepackage{amsmath,amssymb,amsfonts,mathtools}
\usepackage{bm}
\usepackage{booktabs}
\usepackage{array}
\usepackage{enumitem}
\usepackage{xcolor}
\usepackage{hyperref}
\usepackage{aliascnt}
\usepackage[nameinlink,noabbrev]{cleveref}
\hypersetup{hidelinks,pdfauthor={Yitzchak Shmalo},pdftitle={Mesoscopic Rectangular Spikes under Subspace Local Laws}}

\allowdisplaybreaks
\setlist[itemize]{leftmargin=2.1em,itemsep=0.25em,topsep=0.35em}
\setlist[enumerate]{leftmargin=2.2em,itemsep=0.25em,topsep=0.35em}

\newtheorem{theorem}{Theorem}[section]
\newaliascnt{proposition}{theorem}
\newtheorem{proposition}[proposition]{Proposition}
\aliascntresetthe{proposition}
\newaliascnt{lemma}{theorem}
\newtheorem{lemma}[lemma]{Lemma}
\aliascntresetthe{lemma}
\newaliascnt{corollary}{theorem}
\newtheorem{corollary}[corollary]{Corollary}
\aliascntresetthe{corollary}
\newaliascnt{assumption}{theorem}
\newtheorem{assumption}[assumption]{Assumption}
\aliascntresetthe{assumption}
\newaliascnt{definition}{theorem}
\newtheorem{definition}[definition]{Definition}
\aliascntresetthe{definition}
\newaliascnt{remark}{theorem}
\newtheorem{remark}[remark]{Remark}
\aliascntresetthe{remark}
\newaliascnt{example}{theorem}
\newtheorem{example}[example]{Example}
\aliascntresetthe{example}

\crefname{theorem}{Theorem}{Theorems}
\crefname{proposition}{Proposition}{Propositions}
\crefname{lemma}{Lemma}{Lemmas}
\crefname{corollary}{Corollary}{Corollaries}
\crefname{assumption}{Assumption}{Assumptions}
\crefname{definition}{Definition}{Definitions}
\crefname{remark}{Remark}{Remarks}
\crefname{example}{Example}{Examples}
\crefname{section}{Section}{Sections}
\crefname{appendix}{Appendix}{Appendices}
\Crefname{appendix}{Appendix}{Appendices}

\newcommand{\F}{\mathbb F}
\newcommand{\R}{\mathbb R}
\newcommand{\C}{\mathbb C}

\newcommand{\E}{\mathbb E}
\newcommand{\Pp}{\mathbb P}
\newcommand{\Id}{\mathrm I}
\newcommand{\op}{\mathrm{op}}
\newcommand{\tr}{\operatorname{tr}}
\newcommand{\rank}{\operatorname{rank}}
\newcommand{\spec}{\operatorname{spec}}
\newcommand{\dist}{\operatorname{dist}}
\newcommand{\diag}{\operatorname{diag}}
\newcommand{\supp}{\operatorname{supp}}

\newcommand{\St}{\operatorname{St}}
\newcommand{\vague}{\xrightarrow{\mathrm v}}
\newcommand{\weak}{\Rightarrow}
\newcommand{\one}{\mathbf 1}
\newcommand{\cL}{\kappa_{\mathrm L}}
\newcommand{\cR}{\kappa_{\mathrm R}}
\newcommand{\cLR}{\kappa_{\mathrm{LR}}}
\newcommand{\calX}{\mathcal X}
\newcommand{\calY}{\mathcal Y}
\newcommand{\calP}{\mathcal P}
\newcommand{\calW}{\mathcal W}
\newcommand{\calR}{\mathcal R}
\newcommand{\calG}{\mathcal G}
\newcommand{\calF}{\mathcal F}
\newcommand{\calH}{\mathcal H}
\newcommand{\calM}{\mathcal M}
\newcommand{\calT}{\mathcal T}
\newcommand{\calK}{\mathcal K}
\newcommand{\calE}{\mathcal E}
\newcommand{\eps}{\varepsilon}
\newcommand{\dd}{\,\mathrm d}
\newcommand{\ii}{\mathrm i}
\newcommand{\abs}[1]{\left|#1\right|}
\newcommand{\norm}[1]{\left\lVert#1\right\rVert}
\newcommand{\ip}[2]{\left\langle #1,#2\right\rangle}

\title[Mesoscopic rectangular spikes]{Mesoscopic Rectangular Spikes under Subspace Local Laws:\\Outlier Values and Singular Subspaces}
\author{Yitzchak Shmalo}
\thanks{Einstein Institute of Mathematics, Hebrew University of Jerusalem, Jerusalem, Israel.
\emph{Email}: \texttt{yitzchak.shmalo@gmail.com}.}
\subjclass[2020]{Primary 60B20; Secondary 15B52, 62H25}
\keywords{Random matrices, spiked models, BBP transition, singular vectors, local laws, Marchenko--Pastur law}
\date{}

\begin{document}

\begin{abstract}
A rectangular data matrix is often modelled as noise plus a signal of low rank.  When that rank is fixed the picture is classical: a signal direction must exceed a critical strength before it produces a singular value outside the noise bulk, and above the threshold the singular vectors of the observation retain a definite, computable fraction of the planted direction.  We ask what survives when the number of signal directions grows with the dimension.

Everything here rests on one hypothesis, which we call a subspace local law: seen from inside the signal subspace, the resolvent of the noise should look like a scalar.  We show that this hypothesis alone locates the outliers and counts them, and, in a holomorphic form, determines the outlier singular subspaces as well.  The conclusions are stated through spectral projectors rather than individual singular vectors, so they remain meaningful when spike strengths collide or come closer together than the error of the approximation, which at growing rank they must.  We then verify the hypothesis in four genuinely different settings, from deterministic noise viewed through randomly oriented signal directions to independent entries with fixed deterministic ones, and for Marchenko--Pastur noise we compute every constant explicitly.
\end{abstract}

\maketitle

\begingroup
\setcounter{tocdepth}{1}
\tableofcontents
\endgroup

\section{Introduction}\label{sec:intro}

\subsection{The problem}
Suppose you are handed a large rectangular matrix and told that it is mostly noise, but that a few real directions have been planted in it.  Two questions follow at once.  Can you tell from the matrix alone that anything was planted?  And if you can, how much of it do you recover?

For a single planted direction the answers have been known for two decades, and they are sharper than one might expect.  There is a critical strength.  Below it the largest singular value of the observation sits exactly where it would have sat with no signal at all, and the corresponding singular vector is asymptotically orthogonal to the direction that was planted; the signal is not merely hard to see, it is invisible.  Above the threshold an outlier separates from the bulk at a location determined by an explicit map, and the singular vectors keep a definite fraction of the planted direction---a fraction strictly less than one, no matter how much data one has.

This paper is about what happens when the number of planted directions is allowed to grow with the size of the matrix.  The setting is
\begin{equation}\label{eq:model-intro}
    W_n=X_n+P_n\in\F^{n\times m},
    \qquad
    \F\in\{\R,\C\},
    \quad n\le m,
    \quad n/m\to c\in(0,1],
\end{equation}
with $X_n$ noise and $P_n$ a signal of rank $r_n$.  Writing a thin singular-value decomposition
\[
    P_n=U_n\Theta_nV_n^*,\qquad
    \Theta_n=\diag(\theta_{n,1},\ldots,\theta_{n,r_n}),
\]
the questions above become: which strengths $\theta_{n,i}$ push a singular value of $W_n$ out of the bulk, where does it land, and how much do the left and right singular vectors of $W_n$ remember about the columns of $U_n$ and $V_n$---when $r_n$ is no longer fixed but tends to infinity?

At fixed rank all of this is finite-dimensional bookkeeping.  A rank-$r$ perturbation changes the characteristic function of the problem by a determinant of size $2r$, and the outlier condition is that this small determinant vanish.  For rectangular additive deformations that condition is the rectangular $D$-transform of Benaych-Georges and Nadakuditi~\cite{BGN2012}, and their residues give the singular-vector overlaps.  Growing rank breaks the bookkeeping in a specific way.  The determinant is now of size $2r_n\to\infty$, so it is a large matrix rather than a small one, and the entrywise convergence statements that suffice for fixed $r$ say nothing uniform about it.  A single badly behaved direction among $r_n$ of them can move an eigenvalue of that matrix, and one moved eigenvalue is one wrong outlier.

So the question becomes what to assume.  Convergence of the noise spectrum is clearly not enough: it says how the noise behaves on average over all directions, and we need to know how it behaves on the particular growing subspace where the signal lives.  The hypothesis we isolate asks exactly that, and nothing more: \emph{seen from inside the signal subspace, does the resolvent of the noise look like a scalar?}  Made precise on the $2r_n$-dimensional compression of the linearized resolvent, in operator norm and uniformly in the spectral parameter, this is what we call a subspace local law.

Hypotheses of this shape are not new in spirit, and it is worth being precise about the difference.  Isotropic and anisotropic local laws, as developed for sample covariance and Wigner matrices in~\cite{BloemendalEtAl2014,KnowlesYin2017}, control a resolvent bilinear form $\ip p{\calR(\zeta)q}$ for prescribed deterministic vectors $p$ and $q$; the bilinear-form estimates of Hachem--Loubaton--Najim--Vallet~\cite{HachemEtAl2013} do the same for information-plus-noise models.  These are statements about one direction, or one pair of directions, at a time.  What the growing-rank problem needs is a statement about a growing subspace all at once, in operator norm.  The two are related but genuinely different: passing from the first to the second costs a union bound over a net of a sphere whose dimension is itself growing, and whether that is affordable depends on how strong the one-vector tail is.  \Cref{sec:independent-entries} is precisely the work of making that passage, and the answer turns out to be that the strength of the tail, rather than anything about the geometry of the signal, is what limits the admissible rank.

Isolating the hypothesis this way makes the division of labor plain.  The subspace local law is the model-dependent random-matrix estimate; everything after it is a deterministic argument about a finite-dimensional analytic equation whose dimension is allowed to grow.  The tools in that second step---linearization, the matrix determinant lemma, inertia, contour integration---are standard, and we make no claim on them.  The work is in identifying the right growing-dimensional input, proving a transfer theorem that survives repeated and unresolved spikes, and verifying the input for noise that is not rotationally invariant.

\subsection{What is new}
The paper makes four additions to the existing theory.

\paragraph{1. Minimal local hypotheses.}
The outlier transfer theorem itself does not assume $r_n\to\infty$, $r_n=o(n)$, or independence between $X_n$ and $P_n$.  Those conditions are imposed only where they are actually used.  The rank condition $r_n=o(n)$ is needed to leave the global empirical law unchanged, and $r_n\to\infty$ is needed only when one wants a nontrivial empirical distribution of outliers.  Independence appears only in the concrete probabilistic verifications.  Uniform boundedness is used only when one asks for a single global matching statement covering every supercritical spike.  Local value and projector results allow arbitrarily large background spikes under the weighted error condition in \cref{ass:weighted-compatibility}.

\paragraph{2. Outlier values without mutual spike separation.}
On every compact superedge region, the actual outliers are uniformly close to their deterministic $D$-transform locations.  No lower bound is imposed on the gaps between distinct spike strengths.  This is useful when the signal has repeated singular values, a continuous mesoscopic profile, or clusters whose internal gaps are smaller than the error scale.  We formulate the result as a bottleneck comparison of multisets, rather than assigning conceptual meaning to a particular labeling inside an unresolved cluster.

\paragraph{3. Singular vectors and singular subspaces.}
A holomorphic subspace local law gives an operator-norm approximation for the compressed Riesz projector of every separated outlier window.  This is the invariant statement that remains meaningful when spike strengths are repeated or arbitrarily close.  For an isolated spike it reduces to the usual left and right cosine formulas.  For a repeated cluster it shows that the empirical left and right outlier subspaces are obtained from the signal subspace by a common asymptotic rotation and two deterministic shrinkage factors.  We also derive empirical measures weighted by the left and right squared overlaps and by their phase-invariant cross overlap.

\paragraph{4. Non-Haar independent-entry models at growing rank.}
Rotational invariance is only one route to a subspace local law.  We prove that an anisotropic local law, together with product logarithmic-Sobolev concentration, yields the required operator-norm compression for deterministic signal directions throughout $r_n=o(n)$.  This includes Gaussian entries but does not rely on Haar singular vectors.  If only an overwhelming-probability anisotropic local law is used, a direct net argument still yields deterministic signal rank $r_n=O(\log n)$.  Random signal orientations, by contrast, work for arbitrary noise matrices whose singular-value law and top edge converge, including deterministic or correlated noise.

\subsection{Relation with the literature}
The BBP transition originated in the spiked covariance model~\cite{BBP2005}.  The fixed-rank rectangular theory in~\cite{BGN2012} proves convergence of extreme singular values and the projections of their left and right singular vectors.  Its master equation is the starting point for our linearized formulation.  What changes here is that the master equation has dimension $2r_n\to\infty$, so convergence of finitely many bilinear forms must be replaced by an operator-norm subspace law.  Our residue computation recovers the fixed-rank overlap formulas, while the projector formulation remains stable under growing multiplicities.

The closest predecessor is Huang~\cite{Huang2018}, who showed that the finite-rank deformation arguments survive at rank $o(N)$ in the Hermitian world.  He treats both an additive perturbation of a Wigner matrix and a multiplicative one, and obtains outlier rigidity, empirical outlier laws, and eigenvector projections---all without imposing mutual separation among the perturbation eigenvalues.  The growing-dimensional monotonicity and counting philosophy behind our value proof is his; we say below exactly where we borrow it.

Huang's multiplicative model deserves a closer comparison than a citation, because it is already a statement about rectangular matrices in disguise.  A spiked sample covariance matrix $(\Id+P)^{1/2}XX^*(\Id+P)^{1/2}$ is the Gram matrix of the rectangle $(\Id+P)^{1/2}X$, so his eigenvalue conclusions are conclusions about that rectangle's singular values, and his eigenvector conclusions are conclusions about its left singular vectors.  Read that way, three differences with the present paper stand out.  The first is that a multiplicative deformation acts on one side: the right singular vectors carry no planted structure of their own, so there is one overlap to compute rather than two.  Here the signal is $U_n\Theta_nV_n^*$, and $U_n$ and $V_n$ are planted simultaneously and independently; the left and right overlaps are different functions of the same strength, and by \eqref{eq:MP-left-overlap}--\eqref{eq:MP-right-overlap} they coincide only in the square case $c=1$.  The second is the shape of the outlier equation.  A one-sided deformation is governed by a single scalar resolvent; the additive rectangular problem produces the product $d=\phi\widetilde\phi$ of two of them, one for each side of the rectangle, and it is that product---the rectangular $D$-transform---which must be inverted.  The third is what the contour integral returns.  Because the linearization is $2r_n$-dimensional with two blocks, one residue computation yields the left projector, the right projector, and the cross term between them at the same time, which is what makes a basis-invariant treatment of repeated clusters possible.

The mutual-separation point is worth stating against a specific alternative rather than in the abstract.  Liu--Liu--Pan--Zhang--Zhang~\cite{LiuEtAl2025}, who also work at growing rank, assume that the distinct spiked values stay separated from one another and cap the rank at $O(n^{1/3})$; Bao--Ding--Wang~\cite{BaoDingWang2021} work at fixed rank with simple spikes.  Such an assumption is natural when the argument identifies spikes one at a time, since two spikes closer than the error scale cannot be told apart.  We avoid it because our value proof never tries to identify anything.  It counts: the number of singular values of $W_n$ above a level is read off from the inertia of one Hermitian matrix, that count is compared with the count for a deterministic comparison matrix by Weyl's inequality, and the comparison is uniform in the spike index.  A cluster of coincident strengths contributes the right number of outliers to the count whether or not one can say which is which.  What that argument cannot do, and does not claim to do, is attach a label to an individual spike inside an unresolved cluster; we therefore state the conclusion as a bottleneck comparison of multisets, and recover individual vectors only under an explicit isolation hypothesis.

For information-plus-noise localization and exact separation, see Loubaton--Vallet~\cite{LoubatonVallet2011}, Chapon--Couillet--Hachem--Mestre~\cite{ChaponEtAl2014}, and Capitaine~\cite{Capitaine2014}.  Ding~\cite{Ding2020} obtains precise results for deformed rectangular matrices and denoising.  Beyond the separation hypothesis discussed above, the two growing-rank and fixed-rank comparisons differ from us in what they compute: \cite{BaoDingWang2021} obtains distributional limits for singular vectors and subspaces, a second-order question that our first-order statements do not address, and \cite{LiuEtAl2025} works in a heteroskedastic model whose noise is not covered by our scalar hypothesis.  The extensive-rank regime, where the signal changes the global law, has a different character; see, for example, Landau--Mel--Ganguli~\cite{LandauMelGanguli2023}.  Recent work of Forner--Maloney--Rosenow~\cite{FornerMaloneyRosenow2025} studies an extensive family of degenerate rectangular spikes and derives the resulting deformation of the Marchenko--Pastur bulk.  These extensive-rank theories are complementary to the present sublinear-rank setting, in which the global law remains unchanged.

The probabilistic verification for deterministic directions uses isotropic and anisotropic local-law technology for sample covariance matrices and their block linearizations.  In particular, the full block estimate in \cite[Eq.~(1.5) and Thms.~3.6--3.7]{KnowlesYin2017}, together with the outside-spectrum isotropic formulation in \cite[Thms.~2.4--2.5 and Rem.~2.6]{BloemendalEtAl2014}, gives the one-vector input used below after the elementary change of spectral variable $z=\zeta^2$ and a block rescaling.  Bilinear resolvent estimates tailored to information-plus-noise models were developed by Hachem--Loubaton--Najim--Vallet~\cite{HachemEtAl2013}.  Our logarithmic-Sobolev lifting argument then supplies exponentially small tails, and a net upgrades those one-vector estimates to an operator norm on an $o(n)$-dimensional signal space.

\subsection{Where the arguments come from}
Since almost every ingredient below has a source, it seems better to say so once, in one place, than to leave the reader to guess.

The linearized master equation of \cref{lem:master-equation} is the growing-rank form of the determinant identity that underlies every finite-rank spiked calculation; for the rectangular additive model it is the equation of~\cite{BGN2012}, written for a $2r_n$-dimensional compression instead of a fixed one.  The real-axis argument in \cref{sec:transfer}---turn the outlier question into an inertia count, control the count by monotonicity in the spectral parameter, and compare counts by Weyl's inequality---is adapted from Huang~\cite{Huang2018}, who used it for Hermitian perturbations of growing rank; our contribution to it is the weighted square-root comparison of \cref{lem:weighted-square-root}, which is what allows the spike strengths to diverge.  The contour representation of the singular subspaces is the classical Riesz-projector argument, used at fixed rank in~\cite{BGN2012} and in the information-plus-noise literature~\cite{LoubatonVallet2011,ChaponEtAl2014,Capitaine2014}; what is new is running it with a homotopy so that the enclosed-root count is preserved when the equation has growing dimension, and reading all three compressed projectors off the same residue.

On the probabilistic side we borrow rather than build.  The one-vector input in \cref{sec:independent-entries} is the anisotropic local law of Knowles--Yin~\cite{KnowlesYin2017}, with the outside-spectrum isotropic statements of Bloemendal--Erd\H{o}s--Knowles--Yau--Yin~\cite{BloemendalEtAl2014} for the diagonal blocks; the change of spectral variable relating their linearization to ours is elementary and is carried out where it is used.  The lifting from one vector to a growing subspace uses tensorization of the logarithmic-Sobolev inequality and the Herbst argument in the form given by Ledoux~\cite{Ledoux2001}, together with the Bai--Yin edge theorem~\cite{YinBaiKrishnaiah1988}.  The concentration estimates for Haar Stiefel frames in \cref{app:stiefel} are standard sphere and net arguments of the kind collected in Vershynin~\cite{Vershynin2018}.  The smoothing step in \cref{thm:empirical-laws} is the Helffer--Sj\"ostrand functional calculus~\cite{DimassiSjostrand1999}.

\subsection{Organization}
\Cref{sec:setup} introduces the scalar $D$-transform and the real and holomorphic subspace local laws.  The main value, projector, and empirical-measure results are stated in \cref{sec:main}.  Their deterministic proofs occupy \cref{sec:transfer}.  Random signal orientations are treated in \cref{sec:random-orientation}; rotationally invariant noise in \cref{sec:haar-noise}; and independent-entry noise with deterministic directions in \cref{sec:independent-entries}.  The Marchenko--Pastur formulas and a collection of examples appear in \cref{sec:MP}.  Technical concentration estimates are collected in the appendices.

\section{Setup and minimal hypotheses}\label{sec:setup}

\subsection{Signal, linearization, and empirical measures}
For a Hermitian $N\times N$ matrix $A$, write
\[
    \mu_A:=\frac1N\sum_{j=1}^N\delta_{\lambda_j(A)}
\]
for its empirical spectral measure, and write $F_A(x):=\mu_A(({-}\infty,x])$ for its cumulative distribution function.  We keep these two notations separate throughout.

Let
\begin{equation}\label{eq:signal-svd}
    P_n=U_n\Theta_nV_n^*,
    \qquad U_n\in\St(n,r_n),\quad V_n\in\St(m,r_n),
\end{equation}
where $\St(N,r):=\{Q\in\F^{N\times r}:Q^*Q=\Id_r\}$ and
\[
    \Theta_n=\diag(\theta_{n,1},\ldots,\theta_{n,r_n}),
    \qquad \theta_{n,1}\ge\cdots\ge\theta_{n,r_n}>0.
\]
At this point $r_n$ may be fixed, mesoscopic, or even comparable to $n$.  No independence between $P_n$ and $X_n$ is assumed in the deterministic transfer statements.

Introduce the Hermitian linearizations
\begin{equation}\label{eq:linearizations}
    \calX_n=\begin{pmatrix}0&X_n\\X_n^*&0\end{pmatrix},\qquad
    \calP_n=\begin{pmatrix}0&P_n\\P_n^*&0\end{pmatrix},\qquad
    \calW_n=\calX_n+\calP_n.
\end{equation}
Set
\begin{equation}\label{eq:E-K}
    \calE_n=\begin{pmatrix}U_n&0\\0&V_n\end{pmatrix},
    \qquad
    \calK_n=\begin{pmatrix}0&\Theta_n\\\Theta_n&0\end{pmatrix}.
\end{equation}
Then $\calP_n=\calE_n\calK_n\calE_n^*$ and $\calE_n^*\calE_n=\Id_{2r_n}$.  The positive eigenvalues of $\calW_n$ are exactly the positive singular values of $W_n$.

\subsection{The scalar noise law and the rectangular outlier map}
We use the following global assumption only to identify a deterministic edge and a scalar outlier map.

\begin{assumption}[Global noise law]\label{ass:global-noise}
As $n,m\to\infty$ with $n/m\to c\in(0,1]$, there is a compactly supported probability measure $\mu$ on $[0,\infty)$ such that
\begin{equation}\label{eq:global-law}
    \mu_{X_nX_n^*}\weak\mu,
    \qquad \norm{X_n}_{\op}\longrightarrow b:=\sqrt{\sup\supp\mu}
\end{equation}
in probability.
\end{assumption}

For $\zeta\in\C\setminus[-b,b]$, define
\begin{equation}\label{eq:phi-def}
    \phi(\zeta):=\int\frac{\zeta}{\zeta^2-x}\,\mu(\dd x),
    \qquad
    \widetilde\phi(\zeta):=c\phi(\zeta)+\frac{1-c}{\zeta},
    \qquad d(\zeta):=\phi(\zeta)\widetilde\phi(\zeta).
\end{equation}
For real $z>b$, the functions $\phi(z)$, $\widetilde\phi(z)$, and $d(z)$ are positive and strictly decreasing.  We impose the edge condition only for statements involving the BBP threshold.

\begin{assumption}[Finite nonzero edge transform]\label{ass:edge}
The one-sided limit satisfies
\[
    d(b+):=\lim_{z\downarrow b}d(z)\in(0,\infty).
\]
\end{assumption}

Define
\begin{equation}\label{eq:threshold-map}
    \theta_\star:=d(b+)^{-1/2},
    \qquad
    \rho(\theta):=
    \begin{cases}
       d^{-1}(\theta^{-2}),&\theta>\theta_\star,\\
       b,&0\le\theta\le\theta_\star,
    \end{cases}
    \qquad \Lambda(\theta):=\rho(\theta)^2.
\end{equation}
The map $\rho$ is continuous and strictly increasing on $(\theta_\star,\infty)$.

For a supercritical strength $\theta>\theta_\star$, let $\rho=\rho(\theta)$ and define
\begin{equation}\label{eq:overlap-weights}
    \cL(\theta):=-\frac{2\phi(\rho)}{\theta^2d'(\rho)},
    \qquad
    \cR(\theta):=-\frac{2\widetilde\phi(\rho)}{\theta^2d'(\rho)},
    \qquad
    \cLR(\theta):=-\frac{2}{\theta^3d'(\rho)}.
\end{equation}
Since $d'(\rho)<0$ and $d(\rho)=\theta^{-2}$,
\begin{equation}\label{eq:cross-weight}
    \cLR(\theta)=\sqrt{\cL(\theta)\cR(\theta)}.
\end{equation}
We set all three weights equal to zero for $\theta\le\theta_\star$.  Their interpretation as squared cosines and a cross overlap follows from \cref{thm:projectors,cor:isolated-vector}.

\subsection{Subspace local laws}
For $\zeta\notin\spec(\calX_n)$, write
\begin{equation}\label{eq:compressed-resolvent}
    \calR_{X,n}(\zeta):=(\zeta\Id-\calX_n)^{-1},
    \qquad
    \calG_n(\zeta):=\calE_n^*\calR_{X,n}(\zeta)\calE_n.
\end{equation}
The scalar deterministic equivalent is
\begin{equation}\label{eq:M-def}
    \calM(\zeta):=
    \begin{pmatrix}
       \phi(\zeta)\Id_{r_n}&0\\
       0&\widetilde\phi(\zeta)\Id_{r_n}
    \end{pmatrix}.
\end{equation}
The terminology ``local law'' here refers to localization in spectral parameter and in deterministic directions, not necessarily to microscopic spectral scales.  All of our contours stay a fixed positive distance from the limiting noise spectrum.

\begin{assumption}[Real subspace local law]\label{ass:real-SLL}
For every compact interval $I\Subset(b,\infty)$, there is a deterministic $\eps_n(I)\downarrow0$ such that
\begin{equation}\label{eq:real-SLL}
    \Pp\left(
       \norm{X_n}_{\op}<\inf I,
       \quad
       \sup_{z\in I}\norm{\calG_n(z)-\calM(z)}_{\op}\le\eps_n(I)
    \right)\longrightarrow1.
\end{equation}
\end{assumption}

\begin{assumption}[Holomorphic subspace local law]\label{ass:holo-SLL}
For every compact $K\Subset\C\setminus[-b,b]$, there is a deterministic $\eps_n(K)\downarrow0$ such that, with probability tending to one,
\begin{equation}\label{eq:holo-SLL}
    \sup_{\zeta\in K}\norm{\calG_n(\zeta)-\calM(\zeta)}_{\op}\le\eps_n(K).
\end{equation}
\end{assumption}

The holomorphic law implies the real law by restriction to compact intervals outside the edge.  In concrete models it is usually no harder to prove, because the resolvent is smoother away from the real spectrum.

The next compatibility condition is the minimal way to allow unbounded spike strengths.  It can be ignored when $\norm{\Theta_n}_{\op}=O(1)$.

\begin{assumption}[Deformation-weighted compatibility]\label{ass:weighted-compatibility}
On every compact set $K$ used in a theorem, in addition to \eqref{eq:holo-SLL} or \eqref{eq:real-SLL},
\begin{equation}\label{eq:weighted-error}
 \sup_{\zeta\in K}
 \left(
   \norm{(\calG_n(\zeta)-\calM(\zeta))\calK_n}_{\op}
   +\norm{\calK_n(\calG_n(\zeta)-\calM(\zeta))}_{\op}
 \right)\longrightarrow0
\end{equation}
in probability.
\end{assumption}

\begin{remark}[Why this is weaker than bounded spikes]\label{rem:weighted}
If $\norm{\Theta_n}_{\op}\eps_n(K)\to0$, then \cref{ass:weighted-compatibility} follows immediately.  Thus bounded spikes are sufficient, but slowly diverging spikes are also allowed whenever their size is smaller than the reciprocal local-law error.  The weighted formulation is the quantity that actually enters the resolvent equation; no global bound on signal strengths is needed for a local contour theorem.
\end{remark}

\subsection{A map of the examples}
The following table summarizes the verifications proved later.  ``Arbitrary noise'' means that the noise may even be deterministic, provided \cref{ass:global-noise} holds.  The listed rank restrictions are sufficient conditions for the local-law error to vanish; the deterministic transfer theorem itself has no rank restriction.

\begin{table}[ht]
\centering
\small
\renewcommand{\arraystretch}{1.25}
\begin{tabular}{>{\raggedright\arraybackslash}p{0.27\textwidth}>{\raggedright\arraybackslash}p{0.29\textwidth}>{\raggedright\arraybackslash}p{0.15\textwidth}>{\raggedright\arraybackslash}p{0.11\textwidth}}
\toprule
Noise / signal geometry & Main probabilistic input & Allowed rank & Result \\
\midrule
Arbitrary noise; independent Haar left and right signal frames & Stiefel compression concentration & $r_n=o(n)$ & Thm.~\ref{thm:random-orientation} \\
Bi-unitarily invariant noise; arbitrary independent signal frames & Haar singular-vector concentration & $r_n=o(n)$ & Thm.~\ref{thm:haar-noise} \\
Rectangular Gaussian noise; arbitrary independent signal frames & Previous row, or LSI lifting & $r_n=o(n)$ & Cor.~\ref{cor:gaussian-general} \\
Independent entries with a uniform logarithmic-Sobolev inequality; deterministic or independent signal frames & Anisotropic law plus exponential concentration & $r_n=o(n)$ & Thm.~\ref{thm:lsi-entry} \\
Independent entries under a standard overwhelming-probability anisotropic law; deterministic or independent signal frames & Net union bound & $r_n=O(\log n)$ & Cor.~\ref{cor:poly-local-law} \\
Any ensemble with an exponential-tail anisotropic law on the signal space & Direct net argument & $r_n=o(n)$ & Prop.~\ref{prop:exp-tail-gateway} \\
\bottomrule
\end{tabular}
\caption{Concrete routes to the subspace local law.}
\label{tab:examples}
\end{table}

\section{Main results}\label{sec:main}

\subsection{The global empirical law}
The first statement is independent of every local-law assumption.

\begin{theorem}[Sublinear rank does not change the global law]\label{thm:global-law}
Assume \cref{ass:global-noise} and $r_n=o(n)$.  Then
\begin{equation}\label{eq:rank-cdf}
  \norm{F_{W_nW_n^*}-F_{X_nX_n^*}}_\infty\le \frac{2r_n}{n}\longrightarrow0.
\end{equation}
Consequently,
\[
   \mu_{W_nW_n^*}\weak\mu
\]
in probability.  Likewise,
\[
   \mu_{W_n^*W_n}\weak c\mu+(1-c)\delta_0.
\]
\end{theorem}

\subsection{Uniform outlier locations without mutual spacing}
For convenience, set $\sigma_{n+1}(W_n):=0$.  For finite multisets $A=\{a_1\ge\cdots\ge a_k\}$ and $B=\{b_1\ge\cdots\ge b_k\}$ of equal size, write
\[
    d_{\mathrm{bot}}(A,B):=\max_{1\le j\le k}|a_j-b_j|
\]
for their one-dimensional bottleneck distance.  This notation records the optimal monotone pairing but does not assign an intrinsic label inside a cluster.

\begin{theorem}[Uniform local matching without mutual spacing]\label{thm:value-matching}
Assume \cref{ass:global-noise,ass:edge,ass:real-SLL}.  Fix $\delta>0$ and $M>\theta_\star+\delta$, and assume \cref{ass:weighted-compatibility} on a compact interval containing a fixed neighborhood of $\rho([\theta_\star+\delta,M])$.  Define
\[
  I_{n,\delta,M}:=\{i:\theta_\star+\delta\le\theta_{n,i}\le M\},
  \qquad
  k_{n,\delta}:=\#\{i:\theta_{n,i}\ge\theta_\star+\delta\}.
\]
There is a deterministic $\eta_n(\delta,M)\downarrow0$ such that, with probability $1-o(1)$,
\begin{equation}\label{eq:index-location}
  \max_{i\in I_{n,\delta,M}}
  \abs{\sigma_i(W_n)-\rho(\theta_{n,i})}
  \le\eta_n(\delta,M),
\end{equation}
where the maximum over the empty set is zero.  In addition, assuming weighted compatibility on a neighborhood of $\rho(\theta_\star+\delta)$,
\begin{equation}\label{eq:no-extra}
  \sigma_{k_{n,\delta}+1}(W_n)
  \le \rho(\theta_\star+\delta)+\eta_n(\delta,M).
\end{equation}
No upper bound is imposed on spikes outside $[\theta_\star+\delta,M]$.  In particular, stronger spikes may diverge, provided the weighted local-law error still vanishes.  If $\norm{\Theta_n}_{\op}\le M$, then \eqref{eq:index-location} matches every $\delta$-supercritical outlier and is equivalently a bottleneck bound for the two complete supercritical multisets.  No separation is required among their strengths.
\end{theorem}

\begin{remark}[Quantitative rate]\label{rem:value-rate}
Let $\bar q_n\downarrow0$ be a deterministic high-probability envelope for the random comparison error $q_n$ in \eqref{eq:H-close}.  The proof gives
\[
   \eta_n(\delta,M)\le C_{\delta,M}\bar q_n.
\]
For globally bounded spikes and an unweighted subspace-law error $\eps_n$, one may take $q_n\le C_M\eps_n$.  If the compression is first centered at finite-$n$ normalized traces, the uniform scalar trace error is added to this bound.
\end{remark}

\begin{corollary}[Exact separation at fixed levels and windows]\label{cor:exact-separation}
Assume \cref{ass:global-noise,ass:edge,ass:real-SLL} and weighted compatibility on a neighborhood of a fixed $z_0>b$.  Set $\vartheta_0=d(z_0)^{-1/2}$.  If, for some fixed $\gamma>0$,
\[
   \min_{1\le i\le r_n}|\theta_{n,i}-\vartheta_0|\ge\gamma
\]
with probability tending to one, then
\begin{equation}\label{eq:exact-level}
  \#\{j:\sigma_j(W_n)>z_0\}
  =\#\{i:\theta_{n,i}>\vartheta_0\}
\end{equation}
with probability tending to one.  No rank restriction and no upper bound on the spike strengths are needed.  Subtracting the identity at the endpoints gives exact counts on any fixed finite union of superedge intervals whose inverse endpoints stay a fixed distance from the spike strengths.
\end{corollary}

\begin{corollary}[Critical and subcritical spikes do not create macroscopic outliers]\label{cor:no-critical-outlier}
Assume \cref{ass:global-noise,ass:edge,ass:real-SLL}, the weighted compatibility condition on every compact interval in $(b,\infty)$, and $r_n=o(n)$.  If
\[
    \theta_{n,1}\le\theta_\star+o_{\Pp}(1),
\]
then
\[
    \norm{W_n}_{\op}\longrightarrow b
\]
in probability.
\end{corollary}

\subsection{Outlier singular subspaces}
The singular-vector statement is naturally expressed through spectral projectors.  For a Borel set $J\subset(0,\infty)$ that does not meet the singular spectrum at its boundary, let
\begin{align*}
  \widehat\Pi^{\mathrm L}_n(J)
  &:=\sum_{j:\,\sigma_j(W_n)\in J}\widehat u_{n,j}\widehat u_{n,j}^*,\\
  \widehat\Pi^{\mathrm R}_n(J)
  &:=\sum_{j:\,\sigma_j(W_n)\in J}\widehat v_{n,j}\widehat v_{n,j}^*,\\
  \widehat C_n(J)
  &:=\sum_{j:\,\sigma_j(W_n)\in J}\widehat u_{n,j}\widehat v_{n,j}^*.
\end{align*}
These three objects do not depend on the choice of singular-vector basis inside a repeated singular-value block, provided the same rotation is used on the left and right.

For each $J$, define diagonal $r_n\times r_n$ matrices
\begin{align}\label{eq:Kappa-matrices}
  K_{\mathrm L,n}(J)&:=\diag\bigl(\cL(\theta_{n,i})\one_{\{\rho(\theta_{n,i})\in J\}}\bigr)_{i=1}^{r_n},\\
  K_{\mathrm R,n}(J)&:=\diag\bigl(\cR(\theta_{n,i})\one_{\{\rho(\theta_{n,i})\in J\}}\bigr)_{i=1}^{r_n},\\
  K_{\mathrm{LR},n}(J)&:=\diag\bigl(\cLR(\theta_{n,i})\one_{\{\rho(\theta_{n,i})\in J\}}\bigr)_{i=1}^{r_n}.
\end{align}

\begin{definition}[Stable outlier contour]\label{def:stable-contour}
Let $\Gamma$ be a positively oriented finite union of rectifiable Jordan curves, and let $\Omega_\Gamma$ denote the union of their interiors.  We call $\Gamma$ stable for $(\Theta_n)$ if
\[
   \overline{\Omega_\Gamma}\Subset\C\setminus[-b,b],
\]
it encloses a specified subset of the positive scalar roots $\rho(\theta_{n,i})$, and
\begin{equation}\label{eq:contour-stability}
  \sup_n\sup_{\zeta\in\Gamma}
  \norm{\bigl(\Id_{2r_n}-\calM(\zeta)\calK_n\bigr)^{-1}}_{\op}<\infty.
\end{equation}
A compact finite union of intervals $J\Subset(b,\infty)$ is called stable if the deterministic locations stay a fixed positive distance from $\partial J$ and there is a stable contour $\Gamma$ with $\Omega_\Gamma\cap\mathbb R=\operatorname{int}J$ whose enclosed positive roots are exactly those lying in $J$.
\end{definition}

In the scalar models treated below, fixed outlier windows whose boundaries stay a positive distance from the deterministic locations are stable; this is proved in \cref{lem:windows-stable}.  If $\Theta_n$ is random, stability means that one can choose a deterministic contour for which the boundary separation and the inverse bound in \eqref{eq:contour-stability} hold on events of probability tending to one; all conclusions are then understood on those events.  We state this requirement explicitly so that the singular-subspace theorem does not hide a separation condition at the boundary.  Stability permits arbitrary multiplicity and arbitrary spacing inside $J$.

\begin{theorem}[Compressed outlier projectors]\label{thm:projectors}
Assume \cref{ass:global-noise,ass:edge,ass:holo-SLL,ass:weighted-compatibility}.  Let $J\Subset(b,\infty)$ be a stable finite union of intervals.  Then, with probability tending to one, $\partial J$ contains no singular value of $W_n$, the number of singular values of $W_n$ in $J$ equals the number of deterministic locations $\rho(\theta_{n,i})$ in $J$, and
\begin{align}
 \norm{U_n^*\widehat\Pi^{\mathrm L}_n(J)U_n-K_{\mathrm L,n}(J)}_{\op}&\longrightarrow0,\label{eq:left-projector}\\
 \norm{V_n^*\widehat\Pi^{\mathrm R}_n(J)V_n-K_{\mathrm R,n}(J)}_{\op}&\longrightarrow0,\label{eq:right-projector}\\
 \norm{U_n^*\widehat C_n(J)V_n-K_{\mathrm{LR},n}(J)}_{\op}&\longrightarrow0\label{eq:cross-projector}
\end{align}
in probability.

Equivalently, if $\widehat\calP_n(J)$ denotes the positive spectral projector of $\calW_n$ onto the eigenvalues in $J$, then
\begin{equation}\label{eq:full-compressed-projector}
 \calE_n^*\widehat\calP_n(J)\calE_n
 -\frac12
 \begin{pmatrix}
   K_{\mathrm L,n}(J)&K_{\mathrm{LR},n}(J)\\
   K_{\mathrm{LR},n}(J)&K_{\mathrm R,n}(J)
 \end{pmatrix}
 \longrightarrow0
\end{equation}
in operator norm and in probability.
\end{theorem}

\begin{remark}[Why projectors, not labels]\label{rem:projectors-not-labels}
If several spike strengths coincide, individual singular vectors are not identifiable: any common orthogonal or unitary rotation inside the corresponding left/right singular subspaces gives another valid singular-value decomposition.  Even when the strengths are merely closer than the error scale, a vector-by-vector labeling is unstable.  The three operators in \cref{thm:projectors} are basis invariant and remain meaningful in exactly those cases.  Individual vectors are recovered only after an isolation assumption, as in the next corollary.
\end{remark}

\begin{corollary}[Isolated singular-vector overlaps]\label{cor:isolated-vector}
Under the hypotheses of \cref{thm:projectors}, suppose that a stable interval $J$ contains exactly one deterministic location, generated by a simple spike $\theta_{n,i}\to\theta>\theta_\star$, and no other deterministic location.  Let $(\widehat u_n,\widehat v_n)$ be the corresponding unit left and right singular vectors of $W_n$.  Then
\begin{align}
  |\ip{u_{n,i}}{\widehat u_n}|^2&\longrightarrow\cL(\theta),\label{eq:isolated-left}\\
  |\ip{v_{n,i}}{\widehat v_n}|^2&\longrightarrow\cR(\theta),\label{eq:isolated-right}\end{align}
in probability, and the projections onto every other signal direction vanish.  Moreover, without any additional phase convention,
\begin{equation}\label{eq:isolated-cross}
   \ip{u_{n,i}}{\widehat u_n}\,
   \overline{\ip{v_{n,i}}{\widehat v_n}}
   \longrightarrow\cLR(\theta).
\end{equation}
\end{corollary}

\begin{corollary}[Repeated and tight clusters]\label{cor:cluster-vectors}
Assume the hypotheses of \cref{thm:projectors}.  Let $C_n\subset\{1,\ldots,r_n\}$ be a cluster separated from all other deterministic outlier locations by a stable contour, and suppose
\[
   \max_{i\in C_n}|\theta_{n,i}-\theta|\longrightarrow0
\]
for some $\theta>\theta_\star$.  Let $q_n=|C_n|$, and let $\widehat U_{C_n}$ and $\widehat V_{C_n}$ be paired left and right singular-vector matrices for the empirical singular values enclosed by the contour, with the same orthogonal or unitary change of basis applied on the two sides inside every repeated empirical singular-value block.  There exists a random $q_n\times q_n$ orthogonal or unitary matrix $Q_n$ such that
\begin{align}
  \norm{U_{C_n}^*\widehat U_{C_n}-\sqrt{\cL(\theta)}\,Q_n}_{\op}&\longrightarrow0,\label{eq:cluster-left}\\
  \norm{V_{C_n}^*\widehat V_{C_n}-\sqrt{\cR(\theta)}\,Q_n}_{\op}&\longrightarrow0.\label{eq:cluster-right}
\end{align}
Thus the same asymptotic rotation identifies the left and right cluster, while the two sides have different deterministic cosine factors.
\end{corollary}

\subsection{Empirical outlier and overlap measures}
Assume now that $r_n\to\infty$.  Define the spike empirical measure
\begin{equation}\label{eq:Hn}
   H_n:=\frac1{r_n}\sum_{i=1}^{r_n}\delta_{\theta_{n,i}}.
\end{equation}
For a paired singular-vector choice, set
\begin{equation}\label{eq:cross-coefficient}
  \chi_{n,j}:=\tr\bigl(U_n^*\widehat u_{n,j}\widehat v_{n,j}^*V_n\bigr)
  =\sum_{i=1}^{r_n}\ip{u_{n,i}}{\widehat u_{n,j}}
      \overline{\ip{v_{n,i}}{\widehat v_{n,j}}}.
\end{equation}
This scalar is unchanged by the common phase ambiguity of a singular-vector pair.  Inside a repeated empirical singular-value block, the individual coefficients depend on the chosen common basis, but their sum---and hence the atom of the measure below---is invariant.  On the singular-value scale, define
\begin{align}
 \nu_n^{\mathrm{out}}&:=\frac1{r_n}\sum_{j=1}^{n}
        \one_{\{\sigma_j(W_n)>b\}}\delta_{\sigma_j(W_n)},\label{eq:outlier-measure}\\
 \omega_{n}^{\mathrm L}&:=\frac1{r_n}\sum_{j=1}^{n}
        \one_{\{\sigma_j(W_n)>b\}}
        \norm{U_n^*\widehat u_{n,j}}_2^2\,\delta_{\sigma_j(W_n)},\label{eq:left-overlap-measure}\\
 \omega_{n}^{\mathrm R}&:=\frac1{r_n}\sum_{j=1}^{n}
        \one_{\{\sigma_j(W_n)>b\}}
        \norm{V_n^*\widehat v_{n,j}}_2^2\,\delta_{\sigma_j(W_n)},\label{eq:right-overlap-measure}\\
 \omega_{n}^{\mathrm{LR}}&:=\frac1{r_n}\sum_{j=1}^{n}
        \one_{\{\sigma_j(W_n)>b\}}
        \chi_{n,j}\,\delta_{\sigma_j(W_n)}.\label{eq:cross-overlap-measure}
\end{align}
The last measure is generally complex-valued at finite $n$; its limit is a positive measure.  Only test functions supported a positive distance above $b$ are used, so the indicator at the limiting edge causes no ambiguity.

\begin{theorem}[Outlier-value and overlap laws]\label{thm:empirical-laws}
Assume \cref{ass:global-noise,ass:edge,ass:holo-SLL,ass:weighted-compatibility} and $r_n\to\infty$.  Suppose that the restricted supercritical spike measures
\[
  H_n^+:=H_n|_{(\theta_\star,\infty)}
\]
converge vaguely in probability on $(\theta_\star,\infty)$ to a deterministic finite measure $H^+$.  Then, vaguely in probability on $(b,\infty)$,
\begin{align}
  \nu_n^{\mathrm{out}}&\vague \rho_\#H^+,\label{eq:value-measure-limit}\\
  \omega_n^{\mathrm L}&\vague \rho_\#(\cL H^+),\label{eq:left-measure-limit}\\
  \omega_n^{\mathrm R}&\vague \rho_\#(\cR H^+),\label{eq:right-measure-limit}\\
  \omega_n^{\mathrm{LR}}&\vague \rho_\#(\cLR H^+).\label{eq:cross-measure-limit}
\end{align}
Equivalently, for every $f\in C_c((b,\infty))$,
\begin{align*}
  \frac1{r_n}\sum_j f(\sigma_j(W_n))
      &\longrightarrow \int f(\rho(\theta))\,H^+(\dd\theta),\\
  \frac1{r_n}\sum_j f(\sigma_j(W_n))\norm{U_n^*\widehat u_{n,j}}_2^2
      &\longrightarrow \int f(\rho(\theta))\cL(\theta)\,H^+(\dd\theta),\\
  \frac1{r_n}\sum_j f(\sigma_j(W_n))\norm{V_n^*\widehat v_{n,j}}_2^2
      &\longrightarrow \int f(\rho(\theta))\cR(\theta)\,H^+(\dd\theta),\\
  \frac1{r_n}\sum_j f(\sigma_j(W_n))\chi_{n,j}
      &\longrightarrow \int f(\rho(\theta))\cLR(\theta)\,H^+(\dd\theta),
\end{align*}
where the sums may be taken over all singular values because $f$ vanishes near the edge.
\end{theorem}

\begin{remark}[Where the mesoscopic assumptions enter]\label{rem:where-meso}
The value, projector, and empirical transfer theorems do not assume $r_n=o(n)$.  The condition $r_n\to\infty$ appears in \cref{thm:empirical-laws} only because its normalization is designed for a growing collection.  Sublinear rank is used to preserve the global empirical law in \cref{thm:global-law} and in the concrete concentration estimates that verify the subspace law.  This separation prevents the probabilistic rank scale from being built unnecessarily into the deterministic statements.
\end{remark}

\section{Deterministic transfer from a subspace local law}\label{sec:transfer}

Everything in this section is deterministic.  No probability enters after the subspace local law has been assumed, which is the point of isolating it, and the reader who is willing to grant \cref{ass:real-SLL} or \cref{ass:holo-SLL} can read this section without knowing anything about the noise ensemble.

The plan is as follows.  One identity does all the work: a rank-$2r_n$ perturbation changes the characteristic function of the linearization by a determinant of size $2r_n$, and both of our questions are questions about that determinant.  We record it, together with the resolvent and counting identities that come with it, in \cref{lem:master-equation}.  Then we use it twice, in two different places in the complex plane.

On the real axis we do not try to solve the determinant equation at all.  We count.  The number of singular values above a level is the number of eigenvalues of one Hermitian matrix that exceed one, that matrix is close in operator norm to a deterministic one, and Weyl's inequality moves the count across.  Because a count is insensitive to which spike is which, this is where the absence of a spacing hypothesis comes from.  The one thing the argument does need is that the crossing be transverse---that the relevant scalar function pass through the critical level at a definite rate---which is \cref{lem:uniform-slope}.

Off the real axis we do solve it, in the only sense we need.  A homotopy from the deterministic equation to the true one keeps the master matrix invertible along the way, so the argument principle says the number of enclosed roots never changes; and the residue of the compressed deformed resolvent at each root is a $2\times2$ matrix whose entries are exactly the three overlap weights.  Integrating that against the contour turns a statement about roots into a statement about spectral projectors, which is the form that survives repeated spikes.

\subsection{Scalar convergence away from the noise spectrum}
Before any of that, a small piece of bookkeeping.  The two scalar transforms $\phi$ and $\widetilde\phi$ are defined through the limiting measure, but every concrete verification in \cref{sec:random-orientation,sec:haar-noise,sec:independent-entries} produces the finite-$n$ traces instead, so we record once and for all that the two agree in the limit, uniformly on compact sets away from the spectrum.  Define
\begin{align}
 \phi_n(\zeta)&:=\frac1n\tr\left[\zeta(\zeta^2\Id_n-X_nX_n^*)^{-1}\right],\label{eq:phi-n}\\
 \widetilde\phi_n(\zeta)&:=\frac1m\tr\left[\zeta(\zeta^2\Id_m-X_n^*X_n)^{-1}\right]
 =\frac nm\phi_n(\zeta)+\left(1-\frac nm\right)\frac1\zeta.\label{eq:phitilde-n}
\end{align}

\begin{lemma}[Uniform scalar convergence]\label{lem:scalar-convergence}
Under \cref{ass:global-noise}, for every compact $K\Subset\C\setminus[-b,b]$,
\begin{equation}\label{eq:scalar-uniform}
  \sup_{\zeta\in K}|\phi_n(\zeta)-\phi(\zeta)|
  +\sup_{\zeta\in K}|\widetilde\phi_n(\zeta)-\widetilde\phi(\zeta)|
  \longrightarrow0
\end{equation}
in probability.
\end{lemma}

\begin{proof}
On the event that $\norm{X_n}_{\op}$ stays a fixed positive distance from $K$, the functions $x\mapsto \zeta/(\zeta^2-x)$ form a uniformly bounded and equicontinuous family on a common compact interval containing the spectra of $X_nX_n^*$.  Weak convergence of the empirical measures therefore gives uniform convergence on a finite net in $K$, and resolvent Lipschitz bounds extend it to all of $K$.  The identity in \eqref{eq:phitilde-n} gives the second convergence.
\end{proof}

\subsection{Rank inequality and the master equation}
The first lemma is the reason sublinear rank leaves the global picture alone: a low-rank change can move individual eigenvalues wherever it likes, but it cannot move the distribution function by more than the rank.  This is what makes the outliers a genuinely separate question from the bulk.

\begin{lemma}[Rank inequality]\label{lem:rank-inequality}
If $A$ and $B$ are Hermitian $N\times N$ matrices, then
\[
  \norm{F_A-F_B}_\infty\le\frac{\rank(A-B)}N.
\]
Consequently,
\[
 \norm{F_{W_nW_n^*}-F_{X_nX_n^*}}_\infty\le\frac{2r_n}{n}.
\]
\end{lemma}

\begin{proof}
The first statement is the standard interlacing consequence of the min--max principle.  Since
\[
 W_nW_n^*-X_nX_n^*=P_nW_n^*+X_nP_n^*,
\]
the rank of the difference is at most $2r_n$.
\end{proof}

\begin{proof}[Proof of \cref{thm:global-law}]
Apply \cref{lem:rank-inequality} and \cref{ass:global-noise}.  The nonzero eigenvalues of $W_n^*W_n$ and $W_nW_n^*$ agree, while $W_n^*W_n$ has $m-n$ additional zeros.  This gives the stated right-Gram limit.
\end{proof}

The next lemma packages the determinant, counting, and resolvent identities used throughout.

\begin{lemma}[Linearized master equation]\label{lem:master-equation}
Let $\zeta\notin\spec(\calX_n)$ and define
\[
   \calF_n(\zeta):=\Id_{2r_n}-\calG_n(\zeta)\calK_n.
\]
Then
\begin{equation}\label{eq:det-master}
  \det(\zeta\Id-\calW_n)
  =\det(\zeta\Id-\calX_n)\det\calF_n(\zeta).
\end{equation}
Moreover,
\begin{equation}\label{eq:compressed-W-resolvent}
  \calE_n^*(\zeta\Id-\calW_n)^{-1}\calE_n
  =\calF_n(\zeta)^{-1}\calG_n(\zeta)
\end{equation}
whenever either side is defined.

For real $z>\norm{X_n}_{\op}$, $\calG_n(z)$ is positive definite and
\begin{equation}\label{eq:inertia-count}
  \#\{j:\sigma_j(W_n)>z\}
  =\#\left\{\lambda\in\spec\bigl(\calG_n(z)^{1/2}\calK_n\calG_n(z)^{1/2}\bigr):\lambda>1\right\}.
\end{equation}
\end{lemma}

\begin{proof}
The matrix determinant lemma applied to $\calW_n=\calX_n+\calE_n\calK_n\calE_n^*$ gives \eqref{eq:det-master}.  Woodbury's identity gives
\[
 (\zeta\Id-\calW_n)^{-1}
 =(\zeta\Id-\calX_n)^{-1}
 +(\zeta\Id-\calX_n)^{-1}\calE_n\calK_n
 \calF_n(\zeta)^{-1}\calE_n^*(\zeta\Id-\calX_n)^{-1}.
\]
After compression, the identity
\[
  \calG+\calG\calK(\Id-\calG\calK)^{-1}\calG
  =(\Id-\calG\calK)^{-1}\calG
\]
yields \eqref{eq:compressed-W-resolvent}.

For $z>\norm{X_n}_{\op}$, the matrix $z\Id-\calX_n$ is positive definite.  Congruence shows that the number of negative eigenvalues of $z\Id-\calW_n$ equals the number of eigenvalues above one of
\[
  (z\Id-\calX_n)^{-1/2}\calE_n\calK_n\calE_n^*(z\Id-\calX_n)^{-1/2}.
\]
Its nonzero eigenvalues agree with those of
$\calG_n(z)^{1/2}\calK_n\calG_n(z)^{1/2}$.  Positive eigenvalues of $\calW_n$ are the singular values of $W_n$, which proves \eqref{eq:inertia-count}.
\end{proof}

\subsection{A weighted square-root comparison}
The real-axis argument compares two Hermitian matrices.  The following formulation explains why the weighted compatibility condition is the natural one when spike strengths are allowed to grow.

\begin{lemma}[Weighted square-root perturbation]\label{lem:weighted-square-root}
Let $G$ and $M$ be positive definite $2r\times2r$ matrices whose spectra lie in a fixed compact subset of $(0,\infty)$.  Suppose that $M=\diag(a\Id_r,b\Id_r)$ with $a,b>0$, and let
\[
 K=\begin{pmatrix}0&\Theta\\\Theta&0\end{pmatrix}
\]
with $\Theta\ge0$ diagonal.  Then
\begin{equation}\label{eq:weighted-sqrt-bound}
 \norm{G^{1/2}KG^{1/2}-M^{1/2}KM^{1/2}}_{\op}
 \le C\left(
   \norm{(G-M)K}_{\op}+\norm{K(G-M)}_{\op}
 \right),
\end{equation}
where $C$ depends only on the spectral bounds for $G$ and $M$.
\end{lemma}

\begin{proof}
Use the integral representation
\[
 G^{1/2}-M^{1/2}
 =\frac1\pi\int_0^\infty t^{1/2}(t+G)^{-1}(G-M)(t+M)^{-1}\,\dd t.
\]
If $M^\sharp:=\diag(b\Id_r,a\Id_r)$, then
\[
   (t+M)^{-1}K=K(t+M^\sharp)^{-1}.
\]
The uniform spectral bounds and the integral formula therefore give
\[
 \norm{(G^{1/2}-M^{1/2})K}_{\op}\le C\norm{(G-M)K}_{\op}.
\]
Applying the same argument on the other side gives
\[
 \norm{K(G^{1/2}-M^{1/2})}_{\op}\le C\norm{K(G-M)}_{\op}.
\]
Now expand
\begin{align*}
 G^{1/2}KG^{1/2}-M^{1/2}KM^{1/2}
 &=(G^{1/2}-M^{1/2})KG^{1/2}\\
 &\quad+M^{1/2}K(G^{1/2}-M^{1/2})
\end{align*}
and use the boundedness of $G^{1/2}$ and $M^{1/2}$.
\end{proof}

\subsection{Proof of the value theorem}
For real $z>b$, define the scalar crossing functions
\begin{equation}\label{eq:scalar-crossing}
  g_{n,i}(z):=\theta_{n,i}\sqrt{d(z)}.
\end{equation}
The positive eigenvalues of the deterministic comparison matrix
\begin{equation}\label{eq:H0}
  \calH_n^{(0)}(z):=\calM(z)^{1/2}\calK_n\calM(z)^{1/2}
  =\sqrt{d(z)}\begin{pmatrix}0&\Theta_n\\\Theta_n&0\end{pmatrix}
\end{equation}
are exactly $g_{n,1}(z),\ldots,g_{n,r_n}(z)$.  A supercritical deterministic location is the unique solution of $g_{n,i}(z)=1$.

\begin{lemma}[Uniform slope away from the transition]\label{lem:uniform-slope}
Fix $\delta>0$ and $M<\infty$.  On every compact subinterval of $(b,\infty)$ containing $\rho([\theta_\star+\delta,M])$, there is $c_{\delta,M}>0$ such that
\begin{equation}\label{eq:slope}
   g_{n,i}'(z)\le-c_{\delta,M}
\end{equation}
whenever $\theta_{n,i}\in[\theta_\star+\delta,M]$.
\end{lemma}

\begin{proof}
On a compact subset of $(b,\infty)$, $d$ is positive and $d'$ is continuous and strictly negative.  Since
\[
   g_{n,i}'(z)=\frac{\theta_{n,i}d'(z)}{2\sqrt{d(z)}},
\]
the claim follows from the lower bound $\theta_{n,i}\ge\theta_\star+\delta$.
\end{proof}

\begin{proof}[Proof of \cref{thm:value-matching}]
Choose a compact real interval $I\Subset(b,\infty)$ containing a fixed neighborhood of $\rho([\theta_\star+\delta,M])$.  On the high-probability event from the real subspace law, set
\[
   \calH_n(z):=\calG_n(z)^{1/2}\calK_n\calG_n(z)^{1/2}.
\]
By \cref{lem:weighted-square-root} and \cref{ass:weighted-compatibility},
\begin{equation}\label{eq:H-close}
   q_n:=\sup_{z\in I}\norm{\calH_n(z)-\calH_n^{(0)}(z)}_{\op}\longrightarrow0
\end{equation}
in probability.  Weyl's inequality therefore places every ordered eigenvalue of $\calH_n(z)$ within $q_n$ of the corresponding ordered eigenvalue of $\calH_n^{(0)}(z)$, uniformly in $z\in I$.

Fix $i\in I_{n,\delta,M}$ and write $\rho_{n,i}=\rho(\theta_{n,i})$.  By \cref{lem:uniform-slope}, for $L>2/c_{\delta,M}$ and all large $n$,
\[
  g_{n,i}(\rho_{n,i}-Lq_n)>1+q_n,
  \qquad
  g_{n,i}(\rho_{n,i}+Lq_n)<1-q_n.
\]
The inertia identity \eqref{eq:inertia-count} then implies that at least $i$ singular values lie above the first point and at most $i-1$ lie above the second point.  Hence
\[
  |\sigma_i(W_n)-\rho_{n,i}|\le Lq_n.
\]
This count-function argument does not compare neighboring gaps and remains valid when any number of deterministic locations coincide or lie within $o(q_n)$ of one another.

For the no-extra estimate, apply the same slope argument to the envelope $g_{n,\mathrm{env}}(z):=(\theta_\star+\delta)\sqrt{d(z)}$.  After enlarging $L$ if necessary, at $z=\rho(\theta_\star+\delta)+Lq_n$ every scalar crossing with index greater than $k_{n,\delta}$ is below $1-q_n$; the negative scalar eigenvalues are also below one.  Weyl's inequality and \eqref{eq:inertia-count} therefore show that at most $k_{n,\delta}$ singular values exceed this level, proving \eqref{eq:no-extra}.  Taking a deterministic envelope for $Lq_n$ yields $\eta_n(\delta,M)$.
\end{proof}

\begin{proof}[Proof of \cref{cor:exact-separation}]
At the fixed point $z_0$, the positive eigenvalues of the scalar comparison matrix are
\[
   g_{n,i}(z_0)=\frac{\theta_{n,i}}{\vartheta_0}.
\]
By \cref{lem:weighted-square-root} and weighted compatibility,
\[
  \norm{\calH_n(z_0)-\calH_n^{(0)}(z_0)}_{\op}\longrightarrow0
\]
in probability.  On the spike-gap event, every positive scalar eigenvalue is at distance at least $\gamma/\vartheta_0$ from one, while every negative scalar eigenvalue is at distance at least $1$ from one.  Thus the whole scalar spectrum stays at distance at least $\min\{\gamma/\vartheta_0,1\}$ from one.  Weyl's inequality therefore preserves exactly the number of eigenvalues above one.  The inertia identity \eqref{eq:inertia-count} gives \eqref{eq:exact-level}.  Window and finite-union statements follow by subtraction.
\end{proof}

\begin{proof}[Proof of \cref{cor:no-critical-outlier}]
Fix $\alpha>0$ and set $z=b+\alpha$.  Strict monotonicity gives
\[
  \theta_\star^2d(z)<\theta_\star^2d(b+)=1.
\]
Hence, with probability tending to one,
\[
  \theta_{n,1}\sqrt{d(z)}\le1-\gamma
\]
for some $\gamma>0$.  By \cref{lem:weighted-square-root}, the real subspace law, and weighted compatibility at this fixed $z$,
\[
  \norm{\calH_n(z)-\calH_n^{(0)}(z)}_{\op}\longrightarrow0
\]
in probability.  Thus every eigenvalue of $\calH_n(z)$ is below one with probability tending to one, and \eqref{eq:inertia-count} gives $\norm{W_n}_{\op}\le b+\alpha$.  The global law from \cref{thm:global-law} gives $\norm{W_n}_{\op}\ge b-o_{\Pp}(1)$.  Letting $\alpha\downarrow0$ proves the claim.
\end{proof}

\subsection{Contour transfer and Riesz projectors}
The real-axis argument is ideal for locations.  Singular vectors are more naturally encoded by the deformed resolvent on a contour.

\begin{proposition}[Stable contour transfer]\label{prop:contour-transfer}
Assume the holomorphic subspace law and weighted compatibility on a compact neighborhood of a stable contour $\Gamma$.  Let $\Omega_\Gamma$ be the region enclosed by $\Gamma$.  With probability tending to one:
\begin{enumerate}
  \item $\calW_n$ and the scalar master equation
  \begin{equation}\label{eq:scalar-master}
     \det(\Id_{2r_n}-\calM(\zeta)\calK_n)=0
  \end{equation}
  have the same number of eigenvalues or roots in $\Omega_\Gamma$, counted with algebraic multiplicity;
  \item uniformly on $\Gamma$,
  \begin{equation}\label{eq:compressed-resolvent-transfer}
    \calE_n^*(\zeta\Id-\calW_n)^{-1}\calE_n
    -\bigl(\Id_{2r_n}-\calM(\zeta)\calK_n\bigr)^{-1}\calM(\zeta)
    \longrightarrow0
  \end{equation}
  in operator norm.
\end{enumerate}
\end{proposition}

\begin{proof}
Set
\[
  \calF_n^{(0)}(\zeta):=\Id_{2r_n}-\calM(\zeta)\calK_n.
\]
On $\Gamma$,
\[
  \calF_n(\zeta)
  =\calF_n^{(0)}(\zeta)
   \left[\Id_{2r_n}
   -\calF_n^{(0)}(\zeta)^{-1}\bigl(\calG_n(\zeta)-\calM(\zeta)\bigr)\calK_n\right].
\]
Stability and weighted compatibility make the norm of the second term inside brackets $o(1)$, uniformly on $\Gamma$.  The same is true along the homotopy $\calM+t(\calG_n-\calM)$, $0\le t\le1$.  Thus the master matrix remains invertible on $\Gamma$ throughout the homotopy.  The argument principle shows that the number of zeros inside is unchanged.  Since the noise spectrum lies inside a vanishing neighborhood of $[-b,b]$, \eqref{eq:det-master} identifies those zeros with eigenvalues of $\calW_n$ inside $\Gamma$.

For the resolvent, use \eqref{eq:compressed-W-resolvent}.  The resolvent identity gives
\begin{align*}
 \calF_n^{-1}\calG_n-(\calF_n^{(0)})^{-1}\calM
 &=(\calF_n^{-1}-(\calF_n^{(0)})^{-1})\calG_n
   +(\calF_n^{(0)})^{-1}(\calG_n-\calM),\\
 \calF_n^{-1}-(\calF_n^{(0)})^{-1}
 &=\calF_n^{-1}(\calG_n-\calM)\calK_n(\calF_n^{(0)})^{-1}.
\end{align*}
The inverses are uniformly bounded on $\Gamma$, so the unweighted and weighted local-law errors give \eqref{eq:compressed-resolvent-transfer}.
\end{proof}

The scalar matrix in \eqref{eq:scalar-master} splits into independent $2\times2$ blocks.  For a spike $\theta$,
\begin{equation}\label{eq:scalar-block-inverse}
 \biggl[
 \Id_2-
 \begin{pmatrix}\phi&0\\0&\widetilde\phi\end{pmatrix}
 \begin{pmatrix}0&\theta\\\theta&0\end{pmatrix}
 \biggr]^{-1}
 \begin{pmatrix}\phi&0\\0&\widetilde\phi\end{pmatrix}
 =\frac1{1-\theta^2d}
 \begin{pmatrix}
   \phi&\theta d\\
   \theta d&\widetilde\phi
 \end{pmatrix}.
\end{equation}
At the positive root $\rho=\rho(\theta)$, the residue of this matrix is
\begin{equation}\label{eq:residue-block}
  \frac1{-\theta^2d'(\rho)}
  \begin{pmatrix}
    \phi(\rho)&\theta d(\rho)\\
    \theta d(\rho)&\widetilde\phi(\rho)
  \end{pmatrix}
  =\frac12
  \begin{pmatrix}
    \cL(\theta)&\cLR(\theta)\\
    \cLR(\theta)&\cR(\theta)
  \end{pmatrix}.
\end{equation}

That residue is the whole content of the overlap formulas, and it is worth pausing on what it says: the same $2\times2$ matrix carries the left weight, the right weight, and the cross term, and it is symmetric.  The two sides of the rectangle are not being treated separately and then compared---they come out of one computation.

The next lemma is the technical price of doing the contour integral for the empirical laws rather than on a fixed contour.  There we must integrate over a strip that is allowed to approach the real axis, where the resolvent blows up like the reciprocal of the distance, and we need to know that the blow-up is no worse than that \emph{uniformly in the spike strength}---including for strengths so large that the corresponding root has run far away.  The proof is a small dilation trick: the scalar $2\times2$ object is realized as a compression of the resolvent of a genuine self-adjoint operator, after which the bound is the spectral theorem.

\begin{lemma}[Uniform off-axis scalar bounds]\label{lem:off-axis-scalar}
Let $K\Subset(b,\infty)$ be a compact interval.  For $\theta\ge0$, set
\[
 T_\theta(\zeta):=
 \frac1{1-\theta^2d(\zeta)}
 \begin{pmatrix}
   \phi(\zeta)&\theta d(\zeta)\\
   \theta d(\zeta)&\widetilde\phi(\zeta)
 \end{pmatrix}.
\]
There are $y_0,C_K>0$ such that, uniformly for $x\in K$, $0<|y|\le y_0$, and $\theta\ge0$,
\begin{equation}\label{eq:off-axis-scalar-lemma}
  \norm{T_\theta(x+\ii y)}_{\op}\le\frac1{|y|},
  \qquad
  \norm{\left[\Id_2-
  \begin{pmatrix}\phi(x+\ii y)&0\\0&\widetilde\phi(x+\ii y)\end{pmatrix}
  \begin{pmatrix}0&\theta\\\theta&0\end{pmatrix}\right]^{-1}}_{\op}
  \le\frac{C_K}{|y|}.
\end{equation}
\end{lemma}

\begin{proof}
The functions $\phi$ and $\widetilde\phi$ are Stieltjes transforms of probability measures on $[-b,b]$: namely, the symmetric singular-value law associated with $\mu$ and its mixture with an atom at zero.  Let $A_{\mathrm L},A_{\mathrm R}$ be the corresponding multiplication operators and let $e_{\mathrm L},e_{\mathrm R}$ be cyclic unit vectors whose scalar resolvents are $\phi$ and $\widetilde\phi$.  With $A=A_{\mathrm L}\oplus A_{\mathrm R}$ and $E(a,b)=a e_{\mathrm L}\oplus b e_{\mathrm R}$, Woodbury's identity gives
\[
 T_\theta(\zeta)
 =E^*\left(\zeta\Id-
 \left[A+E\begin{pmatrix}0&\theta\\\theta&0\end{pmatrix}E^*\right]
 \right)^{-1}E.
\]
The operator inside the resolvent is self-adjoint, so the spectral theorem gives
$\norm{T_\theta(x+\ii y)}_{\op}\le |y|^{-1}$.

Because $K\Subset(b,\infty)$ and $\phi,\widetilde\phi$ are positive on $K$, after choosing $y_0>0$ small enough there is $c_K>0$ such that
\[
 \min\{|\phi(x+\ii y)|,|\widetilde\phi(x+\ii y)|\}\ge c_K
\]
uniformly for $x\in K$ and $|y|\le y_0$.  The explicit formulas show that the entries of the second matrix in \eqref{eq:off-axis-scalar-lemma} satisfy
\[
 \frac1{1-\theta^2d}=\frac{(T_\theta)_{11}}{\phi}
 =\frac{(T_\theta)_{22}}{\widetilde\phi},\qquad
 \frac{\theta\phi}{1-\theta^2d}=\frac{(T_\theta)_{12}}{\widetilde\phi},\qquad
 \frac{\theta\widetilde\phi}{1-\theta^2d}=\frac{(T_\theta)_{12}}{\phi}.
\]
Thus every entry is bounded by $C_K\norm{T_\theta}_{\op}$, uniformly in $\theta\ge0$, which proves the second estimate.
\end{proof}

The same dilation settles two points left open by \cref{def:stable-contour}: the scalar master equation has no non-real roots, and separated fixed windows are always stable.

\begin{lemma}[Separated windows are stable]\label{lem:windows-stable}
Let $\Gamma$ be a positively oriented finite union of rectifiable Jordan curves with $\overline{\Omega_\Gamma}\Subset\C\setminus[-b,b]$, and suppose that
\[
  \delta_\Gamma:=\inf_n\ \min_{1\le i\le r_n}\ \dist\bigl(\Gamma,\{\pm\rho(\theta_{n,i})\}\bigr)>0,
\]
where strengths with $\theta_{n,i}\le\theta_\star$ contribute no roots and are omitted from the minimum, with the convention $\min\emptyset=+\infty$.  Then the inverse bound \eqref{eq:contour-stability} holds.  Moreover, for every $\theta\ge0$ all zeros of $\zeta\mapsto1-\theta^2d(\zeta)$ in $\C\setminus[-b,b]$ are real: they are exactly the points $\pm\rho(\theta)$, present precisely when $\theta>\theta_\star$, and each is a simple zero because $d'(\rho(\theta))<0$ by \cref{lem:d-monotone}.
\end{lemma}

\begin{proof}
Fix $\theta\ge0$ and let $A_\theta:=A+E\bigl(\begin{smallmatrix}0&\theta\\\theta&0\end{smallmatrix}\bigr)E^*$ be the self-adjoint operator from the proof of \cref{lem:off-axis-scalar}, so that $T_\theta(\zeta)=E^*(\zeta\Id-A_\theta)^{-1}E$ off its spectrum.  Since the perturbation has rank two, Weyl's theorem keeps the essential spectrum inside $[-b,b]$, and a point $\zeta\notin[-b,b]$ lies in $\spec(A_\theta)$ precisely when $1-\theta^2d(\zeta)=0$, by the same determinant computation as in \cref{lem:master-equation}.  Self-adjointness forces such points to be real; on $(b,\infty)$ strict monotonicity of $d$ produces the unique root $\rho(\theta)$ for $\theta>\theta_\star$ and none otherwise, and the roots on $(-\infty,-b)$ are the mirror images because $d$ is even.  This proves the second claim and identifies $\spec(A_\theta)\setminus[-b,b]=\{\pm\rho(\theta)\}$.

For the inverse bound, the spectral theorem gives, for $\zeta\in\Gamma$ and every strength $\theta_{n,i}$,
\[
 \norm{T_{\theta_{n,i}}(\zeta)}_{\op}
 \le\dist\bigl(\zeta,\spec(A_{\theta_{n,i}})\bigr)^{-1}
 \le\max\bigl\{\dist(\Gamma,[-b,b])^{-1},\,\delta_\Gamma^{-1}\bigr\}=:C_\Gamma,
\]
uniformly in $n$, $i$, and the strength.  The functions $\phi$ and $\widetilde\phi$ are continuous and zero-free on $\C\setminus[-b,b]$: they are positive on $(b,\infty)$, odd, and have strictly negative imaginary part in the upper half-plane.  Hence $c_\Gamma:=\min_{\zeta\in\Gamma}\min\{|\phi(\zeta)|,|\widetilde\phi(\zeta)|\}>0$, and the entry identities in the proof of \cref{lem:off-axis-scalar} bound every entry of the block inverse $[\Id_2-\diag(\phi,\widetilde\phi)\bigl(\begin{smallmatrix}0&\theta_{n,i}\\\theta_{n,i}&0\end{smallmatrix}\bigr)]^{-1}$ by $C_\Gamma/c_\Gamma$.  After the pairing permutation, $\Id_{2r_n}-\calM(\zeta)\calK_n$ is the direct sum of these $2\times2$ blocks, so \eqref{eq:contour-stability} holds with constant $2C_\Gamma/c_\Gamma$.
\end{proof}

A stable window in the sense of \cref{def:stable-contour} always produces a contour to which the lemma applies.  Given a compact $J\Subset(b,\infty)$ whose components have endpoints at distance at least $\gamma>0$ from every $\rho(\theta_{n,i})$, take $\Gamma$ to be the boundaries of rectangles of a fixed height $h$ erected on the components of $J$, with $h$ small enough that $\overline{\Omega_\Gamma}\Subset\C\setminus[-b,b]$.  Then $\dist(\Gamma,\rho(\theta_{n,i}))\ge\min\{h,\gamma\}$ for every $i$, and the mirror roots $-\rho(\theta_{n,i})<-b$ are automatically far from $\Gamma$ because $\Gamma$ meets the real axis only in $\partial J\subset(b,\infty)$.  Hence $\delta_\Gamma>0$ and \eqref{eq:contour-stability} holds.

\begin{proof}[Proof of \cref{thm:projectors}]
Let $\Gamma$ be a stable contour enclosing precisely the deterministic locations in $J$.  By \cref{prop:contour-transfer}, the actual and scalar equations have the same number of enclosed roots, and by \cref{lem:windows-stable} the enclosed scalar roots are exactly the real points $\rho(\theta_{n,i})\in J$, counted with multiplicity in $i$.  The positive Riesz projector of $\calW_n$ is
\[
  \widehat\calP_n(J)=\frac1{2\pi\ii}\int_\Gamma(\zeta\Id-\calW_n)^{-1}\,\dd\zeta.
\]
Compressing by $\calE_n$, using \eqref{eq:compressed-resolvent-transfer}, and integrating gives convergence to the sum of the scalar residues.  Formula \eqref{eq:residue-block} yields \eqref{eq:full-compressed-projector}.

If $\widehat U_J$ and $\widehat V_J$ collect the singular vectors with singular values in $J$, then
\[
  \widehat\calP_n(J)=\frac12
  \begin{pmatrix}
    \widehat U_J\widehat U_J^*&\widehat U_J\widehat V_J^*\\
    \widehat V_J\widehat U_J^*&\widehat V_J\widehat V_J^*
  \end{pmatrix}.
\]
Comparing the four blocks proves \eqref{eq:left-projector}--\eqref{eq:cross-projector}.
\end{proof}

\begin{proof}[Proof of \cref{cor:isolated-vector}]
For an isolated simple outlier, each projector in \cref{thm:projectors} has rank one.  The diagonal blocks of \eqref{eq:full-compressed-projector} give \eqref{eq:isolated-left} and \eqref{eq:isolated-right}, including vanishing projections onto the remaining signal directions.  The off-diagonal block gives the phase-invariant product in \eqref{eq:isolated-cross} directly.
\end{proof}

\begin{proof}[Proof of \cref{cor:cluster-vectors}]
Write
\[
 A_n:=U_{C_n}^*\widehat U_{C_n},\qquad
 B_n:=V_{C_n}^*\widehat V_{C_n}.
\]
The projector theorem and continuity of the weights imply
\[
 A_nA_n^*=\cL(\theta)\Id_{q_n}+o_{\op}(1),\quad
 B_nB_n^*=\cR(\theta)\Id_{q_n}+o_{\op}(1),\quad
 A_nB_n^*=\cLR(\theta)\Id_{q_n}+o_{\op}(1).
\]
Since $\cLR=\sqrt{\cL\cR}$, the polar factor of $A_n$ is also an asymptotic polar factor of $B_n$.  Because $\theta>\theta_\star$ we have $\cL(\theta)>0$, so on the event $\norm{A_nA_n^*-\cL(\theta)\Id}_{\op}<\cL(\theta)/2$---which has probability tending to one---every singular value of $A_n$ is at least $\sqrt{\cL(\theta)/2}$.  There $A_n$ is invertible, its polar decomposition $A_n=H_nQ_n$ has $H_n\succ0$ and $Q_n=H_n^{-1}A_n$ genuinely orthogonal or unitary, and $H_n\to\sqrt{\cL(\theta)}\Id$ because the square root is Lipschitz on $[\cL(\theta)/2,2\cL(\theta)]$, uniformly in $q_n$.  Off that event set $Q_n:=\Id$.  The cross relation then gives $B_n-\sqrt{\cR(\theta)}Q_n=o_{\op}(1)$.  This proves \eqref{eq:cluster-left}--\eqref{eq:cluster-right}.
\end{proof}

\subsection{Proof of the empirical laws}

\begin{proof}[Proof of \cref{thm:empirical-laws}]
We first prove the value-measure limit without assuming tightness of the spike measures at infinity.  For $z>b$, put
\[
  N_n(z):=\#\{j:\sigma_j(W_n)>z\},
  \qquad \vartheta(z):=d(z)^{-1/2}.
\]
At this fixed level, weighted square-root comparison and the inertia identity imply that, for every sufficiently small fixed $a>0$, with probability tending to one,
\begin{equation}\label{eq:count-sandwich}
 \#\{i:\theta_{n,i}>\vartheta(z)+a\}
 \le N_n(z)
 \le \#\{i:\theta_{n,i}>\vartheta(z)-a\}.
\end{equation}
Indeed, the scalar positive eigenvalues at $z$ are $\theta_{n,i}/\vartheta(z)$, while the random and scalar comparison matrices differ by $o_{\Pp}(1)$ in operator norm.

Let $b<z_1<z_2$ and write $\vartheta_\ell=\vartheta(z_\ell)$.  Since $z_1>b$ we have $\vartheta_1>\theta_\star$, so the constraint
\[
  a<\min\left\{\frac{\vartheta_2-\vartheta_1}2,\ \vartheta_1-\theta_\star\right\}
\]
leaves both thresholds above $\theta_\star$, where $H_n$ and its supercritical restriction $H_n^+$ agree.  For such $a$, subtracting the two bounds in \eqref{eq:count-sandwich} gives
\begin{align}\label{eq:window-sandwich}
 H_n^+((\vartheta_1+a,\vartheta_2-a])
 &\le \frac1{r_n}\#\{j:z_1<\sigma_j(W_n)\le z_2\}\notag\\
 &\le H_n^+((\vartheta_1-a,\vartheta_2+a])
\end{align}
with probability tending to one.  Choose endpoints that are continuity points of $H^+$ and then let $a\downarrow0$.  Vague convergence of $H_n^+$ on the relatively compact interval between the two thresholds proves convergence of all such window masses.  These windows form a convergence-determining class on $(b,\infty)$, so
\[
   \nu_n^{\mathrm{out}}\vague\rho_\#H^+.
\]
Notice that spikes escaping to infinity cancel in the difference of the two tail counts; no global upper bound is required.

We next treat the overlaps without imposing artificial gaps between nearby locations.  Fix first $f\in C_c^\infty((b,\infty))$, extended by zero to the whole real line, and set
\[
  \calT_n(\zeta):=\calE_n^*(\zeta\Id-\calW_n)^{-1}\calE_n,
  \qquad
  \calT_n^{(0)}(\zeta):=
  (\Id-\calM(\zeta)\calK_n)^{-1}\calM(\zeta).
\]
For $\Im\zeta\ne0$, the exact identity
\begin{equation}\label{eq:off-axis-difference}
 \calT_n(\zeta)-\calT_n^{(0)}(\zeta)
 =(\Id-\calG_n(\zeta)\calK_n)^{-1}
   (\calG_n(\zeta)-\calM(\zeta))
   (\Id-\calK_n\calM(\zeta))^{-1}
\end{equation}
holds.  Work on a fixed compact strip whose real projection contains $\supp f$ and which remains a positive distance from $[-b,b]$.  Let $e_n\downarrow0$ be a deterministic high-probability envelope for the unweighted and weighted local-law errors on that strip.

Applying \cref{lem:off-axis-scalar} block by block gives
\begin{equation}\label{eq:deterministic-off-axis-bound}
  \norm{\calT_n^{(0)}(x+\ii y)}_{\op}
  +\norm{(\Id-\calM(x+\ii y)\calK_n)^{-1}}_{\op}
  +\norm{(\Id-\calK_n\calM(x+\ii y))^{-1}}_{\op}
  \le \frac{C_f}{|y|}
\end{equation}
uniformly in the spike strengths.  When $|y|\ge e_n^{1/2}$, a Neumann-series argument based on weighted compatibility gives the same inverse bound for $\Id-\calG_n\calK_n$.  Hence \eqref{eq:off-axis-difference} yields
\begin{equation}\label{eq:off-axis-comparison-rate}
  \norm{\calT_n(x+\ii y)-\calT_n^{(0)}(x+\ii y)}_{\op}
  \le \frac{C_f e_n}{|y|^2}
  \qquad (|y|\ge e_n^{1/2}).
\end{equation}
For $0<|y|<e_n^{1/2}$, the spectral theorem gives
\[
   \norm{\calT_n(x+\ii y)}_{\op}\le |y|^{-1},
\]
and \eqref{eq:deterministic-off-axis-bound} gives the corresponding deterministic estimate.

Apply the Helffer--Sj\"ostrand functional calculus~\cite{DimassiSjostrand1999} with an almost-analytic extension of order at least three.  In the region $|y|\ge e_n^{1/2}$, the integrand is controlled by $C e_n|y|$ after \eqref{eq:off-axis-comparison-rate}; in the remaining strip it is controlled by $C|y|^2$.  Both integrals tend to zero.  Therefore
\begin{equation}\label{eq:smooth-functional-transfer}
  \norm{\calE_n^*f(\calW_n)\calE_n-f[\calT_n^{(0)}]}_{\op}
  \longrightarrow0
\end{equation}
in probability, where $f[\calT_n^{(0)}]$ denotes the same almost-analytic integral with $\calT_n^{(0)}$ in place of the resolvent.

The scalar comparison is diagonal in the spike index.  Its only positive poles outside $[-b,b]$ are the supercritical roots, and \eqref{eq:residue-block} gives
\begin{equation}\label{eq:deterministic-functional-block}
 f[\calT_n^{(0)}]
 =\frac12\bigoplus_{i=1}^{r_n}
 f(\rho(\theta_{n,i}))
 \begin{pmatrix}
   \cL(\theta_{n,i})&\cLR(\theta_{n,i})\\
   \cLR(\theta_{n,i})&\cR(\theta_{n,i})
 \end{pmatrix},
\end{equation}
where subcritical terms are zero.  On the other hand, the positive spectral decomposition of $\calW_n$ gives
\[
 \calE_n^*f(\calW_n)\calE_n
 =\frac12\sum_j f(\sigma_j(W_n))
 \begin{pmatrix}
  U_n^*\widehat u_{n,j}\widehat u_{n,j}^*U_n&
  U_n^*\widehat u_{n,j}\widehat v_{n,j}^*V_n\\
  V_n^*\widehat v_{n,j}\widehat u_{n,j}^*U_n&
  V_n^*\widehat v_{n,j}\widehat v_{n,j}^*V_n
 \end{pmatrix}.
\]
Take normalized traces of the two diagonal blocks and of the upper-right block in \eqref{eq:smooth-functional-transfer}, and use \eqref{eq:deterministic-functional-block}.  This compares the three empirical overlap integrals with
\[
 \int f(\rho(\theta))\cL(\theta)\,H_n^+(\dd\theta),\qquad
 \int f(\rho(\theta))\cR(\theta)\,H_n^+(\dd\theta),\qquad
 \int f(\rho(\theta))\cLR(\theta)\,H_n^+(\dd\theta),
\]
respectively, up to $o_{\Pp}(1)$.  Vague convergence of $H_n^+$ and continuity of the weights on the compact preimage of $\supp f$ give the three limits for smooth $f$.

Finally, smooth test functions are uniformly dense in $C_c((b,\infty))$.  The left and right overlap measures have total mass at most one.  Moreover,
\[
  \sum_j|\chi_{n,j}|
  \le\frac12\sum_j\left(
     \norm{U_n^*\widehat u_{n,j}}_2^2+
     \norm{V_n^*\widehat v_{n,j}}_2^2
  \right)\le r_n,
\]
so the cross measures have uniformly bounded total variation.  The deterministic weighted measures are likewise uniformly bounded on compact supercritical sets.  Uniform approximation therefore extends the three limits to every $f\in C_c((b,\infty))$.
\end{proof}

\section{Random signal orientations over arbitrary noise}\label{sec:random-orientation}

This is the broadest verification in terms of the noise ensemble.  The noise need not have independent entries or random singular vectors; it may even be deterministic.  Averaging is supplied entirely by the signal directions.

\begin{assumption}[Random signal orientations]\label{ass:random-orientation}
Conditional on $(X_n,\Theta_n)$, the matrices $U_n\in\St(n,r_n)$ and $V_n\in\St(m,r_n)$ are independent Haar Stiefel matrices.  The singular values in $\Theta_n$ may be deterministic or random.
\end{assumption}

\begin{theorem}[Subspace law from random orientations]\label{thm:random-orientation}
Assume \cref{ass:global-noise,ass:random-orientation} and $r_n=o(n)$.  Then \cref{ass:holo-SLL} holds.  More precisely, for every compact $K\Subset\C\setminus[-b,b]$, there is a deterministic $s_n(K)\downarrow0$ such that one may take
\begin{equation}\label{eq:random-orientation-rate}
  \eps_n(K)\le C_K\left(
      \sqrt{\frac{r_n}{n}}+\sqrt{\frac{\log n}{n}}+s_n(K)
  \right).
\end{equation}
Here $s_n(K)$ is a deterministic envelope for the scalar trace error in \cref{lem:scalar-convergence}.  If $\norm{\Theta_n}_{\op}=O(1)$, weighted compatibility follows automatically.
\end{theorem}

\begin{proof}
For $\zeta\notin\spec(\calX_n)$, the block resolvent formula gives
\begin{equation}\label{eq:full-resolvent-blocks}
 (\zeta\Id-\calX_n)^{-1}
 =\begin{pmatrix}
 A_n(\zeta)&B_n(\zeta)\\
 C_n(\zeta)&D_n(\zeta)
 \end{pmatrix},
\end{equation}
where
\begin{align*}
 A_n(\zeta)&=\zeta(\zeta^2\Id_n-X_nX_n^*)^{-1},\\
 D_n(\zeta)&=\zeta(\zeta^2\Id_m-X_n^*X_n)^{-1},\\
 B_n(\zeta)&=X_n(\zeta^2\Id_m-X_n^*X_n)^{-1},\\
 C_n(\zeta)&=X_n^*(\zeta^2\Id_n-X_nX_n^*)^{-1}.
\end{align*}
On the event that $\norm{X_n}_{\op}$ stays a fixed distance from $K$, all four blocks and their first $\zeta$-derivatives are bounded in operator norm by a deterministic constant depending only on $K$.

Condition on $(X_n,\Theta_n)$.  The four blocks of $\calG_n(\zeta)$ are
\[
 U_n^*A_n(\zeta)U_n,
 \quad U_n^*B_n(\zeta)V_n,
 \quad V_n^*C_n(\zeta)U_n,
 \quad V_n^*D_n(\zeta)V_n.
\]
The Stiefel concentration estimates in \cref{lem:stiefel-compression,lem:stiefel-mixed} show, for each fixed $\zeta$, that the diagonal compressions differ from their normalized traces by
$O_{\Pp}(\sqrt{r_n/n}+t)$ and that both mixed compressions are $O_{\Pp}(\sqrt{r_n/n}+t)$, with tails bounded by $C\exp(-cnt^2)$.  A two-dimensional net of $K$, together with the resolvent Lipschitz bound, makes these estimates uniform in $\zeta$ at the additional cost $\sqrt{\log n/n}$.  The normalized traces are $\phi_n$ and $\widetilde\phi_n$, which converge uniformly to $\phi$ and $\widetilde\phi$ by \cref{lem:scalar-convergence}.  This proves \eqref{eq:random-orientation-rate}.
\end{proof}

\begin{corollary}[Outliers and singular subspaces for arbitrary noise]\label{cor:random-orientation-main}
Under \cref{ass:global-noise,ass:edge,ass:random-orientation}, suppose $r_n=o(n)$ and the spike strengths are bounded, or more generally satisfy the weighted compatibility condition induced by \eqref{eq:random-orientation-rate}.  Then the hypotheses of the value and projector transfer theorems are verified.  The empirical laws also apply whenever $r_n\to\infty$ and the restricted supercritical spike measures have the convergence stated in \cref{thm:empirical-laws}.
\end{corollary}

\begin{example}[Deterministic noise with random signal directions]\label{ex:deterministic-noise}
Let $X_n$ be any deterministic sequence whose squared singular-value measures converge to a limit $\mu$ satisfying \cref{ass:edge} and whose largest singular values converge to $b$.  Choose independent Haar frames $U_n,V_n$ and any bounded mesoscopic strength profile $\Theta_n$ with $r_n=o(n)$.  Then the full value and singular-subspace theory applies.  Thus randomness of the noise is not logically necessary; generic relative orientation is enough.
\end{example}

\begin{example}[Correlated and variance-profile noise]\label{ex:correlated-random-orientation}
Suppose $X_n$ has correlated entries, a variance profile, or a deterministic covariance structure for which only the global singular-value law and top edge are known.  Deterministic spike directions may see a matrix-valued resolvent equivalent, but independent Haar signal frames average that structure to the normalized traces.  \Cref{thm:random-orientation} therefore recovers the scalar rectangular $D$-transform without requiring an anisotropic local law for the noise.
\end{example}

\section{Rotationally invariant noise}\label{sec:haar-noise}

The previous section randomized the signal.  We now randomize the singular vectors of the noise and allow arbitrary independent signal directions.

\begin{assumption}[Haar singular vectors]\label{ass:haar-noise}
The pair $(X_n,P_n)$ has an equal-in-law realization such that
\begin{equation}\label{eq:haar-svd}
  X_n=L_n\Sigma_nR_n^*,
\end{equation}
where, conditional on the singular values $\Sigma_n$, the matrix $L_n$ is Haar on $\mathrm O(n)$ or $\mathrm U(n)$, the matrix $R_n$ is Haar on $\St(m,n)$, and $L_n,R_n$ are conditionally independent.  The signal $P_n$ is independent of $(L_n,\Sigma_n,R_n)$.
\end{assumption}

\begin{theorem}[Subspace law for Haar-singular-vector noise]\label{thm:haar-noise}
Assume \cref{ass:global-noise,ass:haar-noise} and $r_n=o(n)$.  Then \cref{ass:holo-SLL} holds with the same rate as in \eqref{eq:random-orientation-rate}.  Consequently, if the limiting law also satisfies \cref{ass:edge}, bounded signals satisfy the value, projector, and cluster theorems; the empirical laws apply under their additional spike-measure hypothesis.
\end{theorem}

\begin{proof}
Conditionally on $R_n$, use auxiliary randomness independent of $(L_n,\Sigma_n,P_n)$ to choose a uniform orthogonal or unitary completion $\widetilde R_n$ whose first $n$ columns are $R_n$.  Then $\widetilde R_n$ is Haar on $\mathrm O(m)$ or $\mathrm U(m)$ and is independent of $(L_n,\Sigma_n,P_n)$.  Conditional on $(\Sigma_n,P_n)$,
\[
   L_n^*U_n\in\St(n,r_n),
   \qquad
   \widetilde R_n^*V_n\in\St(m,r_n)
\]
are independent Haar Stiefel matrices.  In the singular-vector coordinates, the four full resolvent blocks in \eqref{eq:full-resolvent-blocks} are diagonal or rectangular diagonal functions of $\Sigma_n$.  The proof of \cref{thm:random-orientation} therefore applies verbatim after conditioning on $(\Sigma_n,P_n)$.
\end{proof}

\begin{proposition}[Bi-unitary invariance supplies \cref{ass:haar-noise}]\label{prop:orbit-measure}
If the law of $X_n$ is invariant under $X_n\mapsto L X_n R^*$ for all deterministic $L\in\mathrm O(n)$ or $\mathrm U(n)$ and $R\in\mathrm O(m)$ or $\mathrm U(m)$, then the noise admits the Haar singular-vector realization in \cref{ass:haar-noise}.  If $P_n$ is independent of $X_n$, the pair $(X_n,P_n)$ admits that realization while preserving independence.
\end{proposition}

\begin{proof}
Condition on the ordered singular values.  Their orbit under the left-right compact group action is a compact homogeneous space, and it carries a unique invariant probability measure.  The pushforward of independent Haar measure on the two groups is that invariant measure.  Disintegrating the law of $X_n$ over its singular values gives the required equal-in-law representation.  This formulation avoids the artificial phase or sign dependence that can appear in a particular measurable singular-value decomposition.
\end{proof}

\begin{example}[Classes covered by rotational invariance]\label{ex:biunitary-classes}
The theorem includes:
\begin{enumerate}
  \item rectangular real or complex Gaussian matrices;
  \item matrices $L_nT_nR_n^*$ with deterministic singular values $T_n$ and independent Haar/Stiefel frames;
  \item any bi-unitarily invariant ensemble whose density depends only on the singular values, provided its empirical squared singular-value law and top edge converge.
\end{enumerate}
The signal directions may be deterministic, structured, or random, as long as they are independent of the noise singular-vector frames.
\end{example}

\section{Independent entries with deterministic signal directions}\label{sec:independent-entries}

This section gives a concrete non-Haar verification.  The argument has two inputs with different roles.  An anisotropic local law identifies the deterministic center of each resolvent bilinear form.  Logarithmic-Sobolev concentration then improves the tail to an exponential form strong enough to union-bound over an $o(n)$-dimensional net.

\subsection{An exponential-tail gateway}

\begin{proposition}[Exponential-tail anisotropic estimates imply a subspace law]\label{prop:exp-tail-gateway}
Let $K\Subset\C\setminus[-b,b]$.  Suppose that, conditionally on $P_n$ when the signal is random, there are deterministic $a_n(K)\downarrow0$ and constants $C_K,c_K>0$ such that, for every deterministic unit $x,y\in\F^{2r_n}$ and every $t\ge0$,
\begin{equation}\label{eq:exp-tail-aniso}
 \Pp\left(
   \sup_{\zeta\in K}
   \left|\ip{x}{\bigl(\calG_n(\zeta)-\calM(\zeta)\bigr)y}\right|
   >a_n(K)+t
   \ \middle|\ P_n
 \right)
 \le C_K e^{-c_Knt^2}+e^{-c_Kn}.
\end{equation}
Then \cref{ass:holo-SLL} holds with
\begin{equation}\label{eq:exp-gateway-rate}
  \eps_n(K)\le C_K'
  \left(a_n(K)+\sqrt{\frac{r_n}{n}}+\sqrt{\frac{\log n}{n}}\right).
\end{equation}
In particular, the error vanishes for every $r_n=o(n)$.
\end{proposition}

\begin{proof}
Let $\mathcal N_n$ be a $1/4$-net of the unit sphere in $\F^{2r_n}$.  It may be chosen with $|\mathcal N_n|\le9^{4r_n}$.  Apply \eqref{eq:exp-tail-aniso} to all pairs $(x,y)\in\mathcal N_n^2$ and take
\[
  t=C\left(\sqrt{\frac{r_n}{n}}+\sqrt{\frac{\log n}{n}}\right)
\]
with $C$ large.  A union bound makes the simultaneous failure probability tend to zero: the Gaussian-tail term dominates the net cardinality after the displayed choice of $t$, while $9^{O(r_n)}e^{-c_Kn}=o(1)$ because $r_n=o(n)$.  The standard net estimate for operator norms increases the bound by only an absolute factor.
\end{proof}

\subsection{Logarithmic-Sobolev independent entries}

\begin{assumption}[Independent entries with LSI concentration]\label{ass:lsi-entries}
Let
\begin{equation}\label{eq:iid-model}
   X_n=\frac1{\sqrt m}(\xi_{ij})_{1\le i\le n,\,1\le j\le m},
\end{equation}
where, in the real case, the entries are independent copies of a fixed centered variance-one random variable.  In the complex case, their real and imaginary parts are independent copies of fixed centered laws, each of variance $1/2$, so that $\E|\xi_{ij}|^2=1$.  Each one-dimensional coordinate law satisfies a logarithmic-Sobolev inequality with a finite constant.  The signal $P_n$ is deterministic or independent of $X_n$.
\end{assumption}

The LSI assumption is a concentration hypothesis, not a rotational-invariance hypothesis.  It is satisfied by Gaussian variables, by uniformly strongly log-concave laws, and by bounded perturbations of such laws through the Holley--Stroock principle~\cite{HolleyStroock1987}.  The fixed-law formulation lets us invoke the classical Marchenko--Pastur and Bai--Yin theorems directly.  The proof extends without change to triangular arrays whenever the global Marchenko--Pastur law, convergence of the top singular value, and the same uniform product concentration estimate are available.  More generally, LSI may be replaced by any product concentration inequality that yields the tail in \eqref{eq:lsi-bilinear-tail} below.

\begin{lemma}[Resolvent bilinear forms have exponential tails]\label{lem:lsi-resolvent-tail}
Under \cref{ass:lsi-entries}, let $K\Subset\C\setminus[-(1+\sqrt c),1+\sqrt c]$.  There is a deterministic $a_n(K)\downarrow0$ such that for every deterministic unit vectors $p,q\in\F^{n+m}$,
\begin{equation}\label{eq:lsi-bilinear-tail}
 \Pp\left(
   \sup_{\zeta\in K}
   \left|\ip{p}{\bigl(\calR_{X,n}(\zeta)-\Pi(\zeta)\bigr)q}\right|
   >a_n(K)+t
 \right)
 \le C_Ke^{-c_Knt^2}+e^{-c_Kn},
\end{equation}
where
\[
  \Pi(\zeta)=
  \begin{pmatrix}
    \phi_{\mathrm{MP},c}(\zeta)\Id_n&0\\
    0&\widetilde\phi_{\mathrm{MP},c}(\zeta)\Id_m
  \end{pmatrix}.
\]
\end{lemma}

\begin{proof}
We separate centering from concentration and give the details needed for the growing-dimensional union bound.

\emph{Step 1: a uniform good set.}  Set $b=1+\sqrt c$ and $\kappa_0:=\dist(K,[-b,b])>0$.  Set $B:=b+\kappa_0/8$ and
\[
   \kappa:=\dist(K,[-B,B])\ge \kappa_0/2,
\]
and let $\mathcal A_n:=\{X:\norm X_{\op}\le B\}$.  The fixed coordinate laws and their LSI bounds imply subgaussian tails of all coordinates.  The Bai--Yin edge theorem gives $\norm{X_n}_{\op}\to b$, and the concentration estimate below makes the operator norm uniformly integrable, so $\E\norm{X_n}_{\op}\to b$; see \cite{YinBaiKrishnaiah1988,Vershynin2018}.  Moreover, as a function of the unnormalized real coordinates $(\xi_{ij})$, the map
\[
   (\xi_{ij})\longmapsto
   \norm{m^{-1/2}(\xi_{ij})}_{\op}
\]
is $m^{-1/2}$-Lipschitz in Euclidean norm.  Product LSI and the Herbst argument~\cite{Ledoux2001} therefore imply, for the fixed $B>b$ and all large $n$,
\begin{equation}\label{eq:good-norm-tail}
   \Pp(X_n\notin\mathcal A_n)
   \le 2\exp\!\left[-c m\bigl(B-\E\norm{X_n}_{\op}\bigr)^2\right]
   \le e^{-c_Kn}.
\end{equation}

\emph{Step 2: a dimension-free Lipschitz bound.}  For $X,Y\in\mathcal A_n$, the resolvent identity gives
\begin{align*}
 &\left|\ip p{(\zeta\Id-\calX)^{-1}q}
       -\ip p{(\zeta\Id-\calY)^{-1}q}\right|\\
 &\qquad\le
   \norm{(\zeta\Id-\calX)^{-1}}_{\op}
   \norm{\calX-\calY}_{\op}
   \norm{(\zeta\Id-\calY)^{-1}}_{\op}
   \le \frac{\sqrt2}{\kappa^2}\norm{X-Y}_{\mathrm F}.
\end{align*}
Thus the real and imaginary parts of the bilinear resolvent form are $C_K$-Lipschitz on $\mathcal A_n$ in Frobenius norm.  Extend those two parts by McShane's theorem, combine them into a complex-valued map, and compose with the Euclidean projection onto the closed disk of radius $\kappa^{-1}$.  This produces a globally defined complex random variable $\widetilde F_{\zeta,p,q}$, still $C_K$-Lipschitz after changing the constant, which agrees with the original bilinear form on $\mathcal A_n$ and is deterministically bounded.  Since $X_n=m^{-1/2}(\xi_{ij})$, these extensions are $C_K/\sqrt m$-Lipschitz functions of the unnormalized real coordinates.

Tensorization of the logarithmic-Sobolev inequality and the Herbst argument~\cite{Ledoux2001} now give, for each fixed $\zeta,p,q$,
\begin{equation}\label{eq:fixed-z-lsi}
  \Pp\bigl(|\widetilde F_{\zeta,p,q}-\E\widetilde F_{\zeta,p,q}|>t\bigr)
  \le4e^{-c_Kmt^2}
  \le4e^{-c_K'nt^2}.
\end{equation}
For the original resolvent form, the same estimate holds after adding the exceptional probability in \eqref{eq:good-norm-tail}.

\emph{Step 3: centering on a polynomial mesh.}
The full block anisotropic local law of Knowles--Yin
\cite[Eq.~(1.5) and Thms.~3.6--3.7]{KnowlesYin2017}, applied to sample covariance matrices with identity population covariance, controls all four blocks of the linearized resolvent simultaneously and implies the following.  For every $D>0$ there are a polynomial mesh $\mathcal K_n\subset K$ and $b_n(K)\downarrow0$ such that, uniformly for deterministic unit vectors $p,q$ and $\zeta\in\mathcal K_n$,
\[
  \Pp\left(
    \left|\ip p{(\calR_{X,n}(\zeta)-\Pi(\zeta))q}\right|>b_n(K)
  \right)\le n^{-D}.
\]
The earlier isotropic Gram-resolvent results of Bloemendal--Erd\H{o}s--Knowles--Yau--Yin~\cite[Thms.~2.4--2.5 and Rem.~2.6]{BloemendalEtAl2014} provide the corresponding diagonal-block estimates, but the mixed block used here is taken from the full linearized law in~\cite{KnowlesYin2017}.
To relate the cited block linearization to ours, put $z=\zeta^2$.  With $G_{\mathrm{KY}}(z)$ denoting the linearization used in \cite{KnowlesYin2017}, one has, up to the harmless convention for the sign of the resolvent,
\[
 \calR_{X,n}(\zeta)
 =-\begin{pmatrix}\zeta^{-1}\Id_n&0\\0&\Id_m\end{pmatrix}
   G_{\mathrm{KY}}(\zeta^2)
   \begin{pmatrix}\Id_n&0\\0&\zeta\Id_m\end{pmatrix}.
\]
All scaling factors are uniformly bounded on $K$.  The cited theorem is naturally centered at the deterministic equivalent with the finite aspect ratio $c_n=n/m$; its uniform difference from $\Pi(\zeta)$ is $o(1)$ on $K$ because $c_n\to c$, and this deterministic error is absorbed into $b_n(K)$.  If $K$ meets the real axis outside the limiting support, the outside-spectrum formulation \cite[Thm.~3.7]{KnowlesYin2017}, or \cite[Thm.~2.5 and Rem.~2.6]{BloemendalEtAl2014} for the diagonal blocks, gives the same mesh estimate at the real points.  Alternatively, evaluate the law at $\zeta_0=x+\ii n^{-L}$: on the event $\{\norm{X_n}_{\op}\le B\}$ both resolvents are analytic in a $\kappa$-neighborhood of $x$, so the resolvent identity gives $|\ip p{(\calR_{X,n}(x)-\calR_{X,n}(\zeta_0))q}|\le\kappa^{-2}n^{-L}$, and $\Pi$ moves by the same order.

Fix $D$ large.  At every mesh point, on the intersection of the local-law event with $\mathcal A_n$, the extension equals the original bilinear form.  Since the clipped extensions are deterministically bounded, \eqref{eq:good-norm-tail} gives
\[
  \sup_{\|p\|=\|q\|=1}\sup_{\zeta\in\mathcal K_n}
  \left|\E\widetilde F_{\zeta,p,q}-\ip p{\Pi(\zeta)q}\right|
  \le b_n(K)+C_Kn^{-D}+C_Ke^{-c_Kn}.
\]
Choose the mesh spacing to be $n^{-L}$ for a fixed sufficiently large $L$.  For a general $\zeta\in K$, let $\zeta_0\in\mathcal K_n$ be a nearest mesh point.  Although the McShane extensions at different spectral parameters need not be jointly Lipschitz, on $\mathcal A_n$ they both equal the original resolvent forms.  Hence
\begin{align*}
 \left|\E\widetilde F_{\zeta,p,q}-\E\widetilde F_{\zeta_0,p,q}\right|
 &\le C_K|\zeta-\zeta_0|+C_K\Pp(\mathcal A_n^c),
\end{align*}
and $\Pi$ is uniformly Lipschitz on $K$.  It follows that
\begin{equation}\label{eq:uniform-centering}
  \sup_{\|p\|=\|q\|=1}\sup_{\zeta\in K}
  \left|\E\widetilde F_{\zeta,p,q}-\ip p{\Pi(\zeta)q}\right|
  \le a_n^{(0)}(K)
\end{equation}
for a deterministic $a_n^{(0)}(K)\downarrow0$.

\emph{Step 4: uniformity in $\zeta$.}
Set $\tau_n=C_K\sqrt{\log n/n}$, with $C_K$ large, and put
\[
  T_K:=1+\kappa^{-1}+\sup_{\zeta\in K}\norm{\Pi(\zeta)}_{\op}.
\]
First suppose $\tau_n\le t\le T_K$.  Choose a mesh $\mathcal K(t)$ of $K$ with spacing at most $c_Kt$; its cardinality is $O_K(1+t^{-2})$.  Apply \eqref{eq:fixed-z-lsi} to the extensions only at the mesh points.  Since $nt^2\ge C_K^2\log n$, the mesh cardinality is absorbed into the exponential tail, and with probability at least
\[
  1-C_Ke^{-c_Knt^2}-e^{-c_Kn}
\]
all mesh-point deviations from their expectations are at most $t/2$.  On $\mathcal A_n$, the extensions at the mesh points equal the original resolvent forms.  The resolvent derivative bound $\norm{\partial_\zeta\calR_{X,n}(\zeta)}_{\op}\le\kappa^{-2}$ and the corresponding Lipschitz bound for $\Pi$ then fill the gaps between mesh points.  Together with \eqref{eq:uniform-centering}, this yields
\[
 \sup_{\zeta\in K}
 \left|\ip p{(\calR_{X,n}(\zeta)-\Pi(\zeta))q}\right|
 \le a_n^{(0)}(K)+t
\]
outside the displayed exceptional event.  For $0\le t<\tau_n$, use the estimate at $t=\tau_n$ and monotonicity of the tail.  If $t>T_K$, then on $\mathcal A_n$ the left-hand side is at most $\kappa^{-1}+\sup_{\zeta\in K}\norm{\Pi(\zeta)}_{\op}<T_K$, so the event is empty; \eqref{eq:good-norm-tail} controls the complement.  Thus, with
\[
  a_n(K):=a_n^{(0)}(K)+\tau_n,
\]
we obtain \eqref{eq:lsi-bilinear-tail} for every $t\ge0$.
\end{proof}

\begin{theorem}[Independent-entry subspace law at rank $o(n)$]\label{thm:lsi-entry}
Assume \cref{ass:lsi-entries} and $r_n=o(n)$.  Then the global noise law holds with squared singular-value limit $\mathrm{MP}_c$ and edge $b=1+\sqrt c$, the edge transform is finite and nonzero, and the holomorphic subspace local law holds with the Marchenko--Pastur deterministic equivalent.  More precisely,
\begin{equation}\label{eq:lsi-SLL-rate}
  \eps_n(K)\le C_K\left(
    a_n(K)+\sqrt{\frac{r_n}{n}}+\sqrt{\frac{\log n}{n}}
  \right)
\end{equation}
for every compact $K$ outside the limiting noise spectrum.  Consequently, the value and singular-subspace theorems apply to every bounded deterministic signal and to every bounded signal independent of the noise; the empirical laws apply when their stated spike-measure hypothesis is also satisfied.
\end{theorem}

\begin{proof}
The Marchenko--Pastur law and convergence of the top singular value follow from the classical i.i.d.\ theory; the fixed LSI coordinate laws are subgaussian~\cite{MarchenkoPastur1967,YinBaiKrishnaiah1988,BaiSilverstein2010,Vershynin2018}.  Condition on $P_n$ if it is random.  Apply \cref{lem:lsi-resolvent-tail} to vectors of the form $\calE_nx$ and $\calE_ny$, and invoke \cref{prop:exp-tail-gateway}.  The Marchenko--Pastur calculation in \cref{prop:MP-transform} verifies the edge condition.
\end{proof}

\begin{remark}[The genuinely random-matrix step]\label{rem:genuine-rmt}
The transfer from a subspace law to the outlier equation is deterministic.  In \cref{thm:lsi-entry}, however, the subspace law is proved for a non-Haar ensemble and a growing deterministic signal space.  The anisotropic local law alone controls one prescribed pair of vectors; the LSI argument supplies a tail strong enough to control exponentially many pairs simultaneously.  This is the probabilistic step that allows the signal rank to be any $o(n)$ rather than merely fixed or logarithmic.
\end{remark}

\subsection{What follows from an overwhelming-probability local law alone}

\begin{corollary}[Logarithmic rank from polynomial tails]\label{cor:poly-local-law}
Suppose that for every compact $K\Subset\C\setminus[-b,b]$ and every $D>0$, an anisotropic local law gives a deterministic $a_n(K)\downarrow0$ such that
\begin{equation}\label{eq:overwhelming-law}
 \sup_{\|p\|=\|q\|=1}
 \Pp\left(
   \sup_{\zeta\in K}
   \left|\ip{p}{(\calR_{X,n}(\zeta)-\Pi(\zeta))q}\right|>a_n(K)
 \right)\le n^{-D}
\end{equation}
for all large $n$, where the supremum means that the same bound holds for every deterministic pair.  Assume that the signal frames $U_n,V_n$ are deterministic, or independent of $X_n$ so that one may condition on them.  If $r_n\le A\log n$ for a fixed $A<\infty$, then the holomorphic subspace local law holds on the signal space.
\end{corollary}

\begin{proof}
Use a $1/4$-net of the unit sphere in $\F^{2r_n}$.  Its pair count is at most $9^{8r_n}\le n^{C A}$.  Choose $D>CA+2$ in \eqref{eq:overwhelming-law}, take a union bound, and pass from the net to the operator norm.
\end{proof}

\begin{remark}[Independent entries beyond LSI]\label{rem:beyond-lsi}
Standard independent-entry sample covariance ensembles with sufficiently many moments satisfy anisotropic local laws of the form \eqref{eq:overwhelming-law}.  Therefore deterministic signals of rank $O(\log n)$ are covered without an LSI assumption.  Any improvement of the one-vector tail immediately enlarges the admissible rank through \cref{prop:exp-tail-gateway}; the transfer theorems themselves require no change.
\end{remark}

\section{The Marchenko--Pastur model and explicit examples}\label{sec:MP}

\subsection{Explicit transforms, locations, and overlaps}
Let $X_n$ have independent real $N(0,1/m)$ or complex $\mathcal N_{\C}(0,1/m)$ entries.  More generally, the formulas below apply whenever the limiting squared singular-value law is the Marchenko--Pastur law $\mathrm{MP}_c$.

Set
\begin{equation}\label{eq:MP-edges}
  a_\pm=(1\pm\sqrt c)^2,
  \qquad b=1+\sqrt c,
\end{equation}
and, for $z>b$,
\begin{equation}\label{eq:Delta-MP}
  \Delta(z):=\sqrt{(z^2-a_-)(z^2-a_+)},
\end{equation}
with the positive branch.

\begin{proposition}[Marchenko--Pastur $D$-transform]\label{prop:MP-transform}
For $z>b$,
\begin{align}
 \phi(z)&=\frac{z^2+c-1-\Delta(z)}{2cz},\label{eq:MP-phi}\\
 \widetilde\phi(z)&=\frac{z^2-c+1-\Delta(z)}{2z},\label{eq:MP-phitilde}\\
 d(z)&=\frac{2}{z^2-(1+c)+\Delta(z)}.\label{eq:MP-d}
\end{align}
Consequently,
\begin{equation}\label{eq:MP-threshold}
  d(b+)=\frac1{\sqrt c},
  \qquad \theta_\star=c^{1/4}.
\end{equation}
For $\theta>c^{1/4}$,
\begin{align}
 \rho(\theta)^2
 &=\theta^2+(1+c)+\frac c{\theta^2}
 =\frac{(1+\theta^2)(c+\theta^2)}{\theta^2},\label{eq:MP-location}\\
 \cL(\theta)
 &=\frac{1-c/\theta^4}{1+c/\theta^2},\label{eq:MP-left-overlap}\\
 \cR(\theta)
 &=\frac{1-c/\theta^4}{1+1/\theta^2},\label{eq:MP-right-overlap}\\
 \cLR(\theta)
 &=\sqrt{\frac{1-c/\theta^4}{1+c/\theta^2}}
   \sqrt{\frac{1-c/\theta^4}{1+1/\theta^2}}.\label{eq:MP-cross-overlap}
\end{align}
\end{proposition}

\begin{proof}
The Stieltjes transform of $\mathrm{MP}_c$ gives \eqref{eq:MP-phi}, and \eqref{eq:MP-phitilde} follows from
$\widetilde\phi=c\phi+(1-c)/z$.  Multiplication gives \eqref{eq:MP-d}.  At $z=b$, $\Delta(b)=0$, which yields \eqref{eq:MP-threshold}.

If
\[
  \lambda=\theta^2+(1+c)+\frac c{\theta^2},
\]
then
\[
  \sqrt{(\lambda-a_-)(\lambda-a_+)}
  =\theta^2-\frac c{\theta^2}
\]
for $\theta>c^{1/4}$.  Substitution into \eqref{eq:MP-d} gives $d(\sqrt\lambda)=\theta^{-2}$ and hence \eqref{eq:MP-location}.  Differentiating \eqref{eq:MP-d} and substituting into \eqref{eq:overlap-weights} gives \eqref{eq:MP-left-overlap}--\eqref{eq:MP-cross-overlap}.
\end{proof}

\begin{corollary}[Gaussian mesoscopic signals]\label{cor:gaussian-general}
Let $X_n$ be rectangular real or complex Gaussian noise, and let $P_n$ be deterministic or independent of $X_n$.  If $r_n=o(n)$ and the spike strengths are bounded, then the value, projector, isolated-vector, and cluster theorems hold with the explicit formulas \eqref{eq:MP-location}--\eqref{eq:MP-cross-overlap}.  The empirical laws hold as well when $r_n\to\infty$ and their stated spike-measure convergence is satisfied.
\end{corollary}

\begin{proof}
Gaussian noise is bi-unitarily invariant, so \cref{thm:haar-noise} applies.  Alternatively, Gaussian coordinates satisfy the LSI hypothesis in \cref{thm:lsi-entry}.
\end{proof}

\subsection{Examples of admissible mesoscopic signals}
The examples below are stated for Marchenko--Pastur noise for concreteness, but the same constructions apply to any limiting law satisfying \cref{ass:edge}.

\begin{example}[A repeated mesoscopic block]\label{ex:repeated-block}
Fix $\alpha\in(0,1)$ and $\theta_0>c^{1/4}$, and let
\[
  r_n=\lfloor n^\alpha\rfloor,
  \qquad \Theta_n=\theta_0\Id_{r_n}.
\]
Then $W_n$ has $r_n$ outliers at
\[
  \rho(\theta_0)+o_{\Pp}(1)
\]
uniformly.  Fix any stable interval $J\Subset(b,\infty)$ containing $\rho(\theta_0)$ and no other deterministic outlier location.  Individual singular vectors inside this block are not identifiable, but the empirical left and right outlier subspaces satisfy
\begin{align*}
 \norm{U_n^*\widehat\Pi_n^{\mathrm L}(J)U_n-\cL(\theta_0)\Id_{r_n}}_{\op}&\to0,\\
 \norm{V_n^*\widehat\Pi_n^{\mathrm R}(J)V_n-\cR(\theta_0)\Id_{r_n}}_{\op}&\to0.
\end{align*}
The common-rotation conclusion of \cref{cor:cluster-vectors} is the correct vector statement for the repeated block.
\end{example}

\begin{example}[Two unresolved clusters]\label{ex:two-clusters}
Let $r_n=\lfloor n^{2/3}\rfloor$, choose $p\in(0,1)$, and let the multiset of spike strengths consist of
\[
 \theta_2+n^{-1/3}b_{n,i}
 \quad(1\le i\le\lfloor pr_n\rfloor),
 \qquad
 \theta_1+n^{-1/3}a_{n,i}
 \quad(\lfloor pr_n\rfloor<i\le r_n),
\]
where $\theta_2>\theta_1>c^{1/4}$ and the arrays $a_{n,i},b_{n,i}$ are uniformly bounded; arrange the resulting strengths in nonincreasing order, as required by the standing convention.  No spacing is assumed within either cluster.  For large $n$, fixed disjoint windows around $\rho(\theta_1)$ and $\rho(\theta_2)$ are stable by \cref{lem:windows-stable}, so each cluster produces the correct number of outliers near $\rho(\theta_1)$ or $\rho(\theta_2)$, and its left/right subspaces have the weights $\cL(\theta_j)$ and $\cR(\theta_j)$.
\end{example}

\begin{example}[A continuous strength profile]\label{ex:continuous-profile}
Let $r_n\to\infty$, $r_n=o(n)$, and choose deterministic strengths whose empirical measure converges to a probability law $H$ supported on
$[c^{1/4}+\delta,M]$.  Then
\[
  \nu_n^{\mathrm{out}}\vague\rho_\#H,
\]
and the left/right overlap measures converge to
\[
  \rho_\#(\cL H),
  \qquad
  \rho_\#(\cR H).
\]
This describes simultaneously a continuum of outlier locations and a continuum of asymptotic estimation accuracies.
\end{example}

\begin{example}[A mixed subcritical and supercritical cloud]\label{ex:mixed-cloud}
Suppose $H_n\weak H$ on $[0,M]$, where $H$ may place arbitrary mass below or at $c^{1/4}$ and has a nonzero restriction above $c^{1/4}$.  Only the restricted measures $H_n^+$ need converge vaguely for the superedge conclusions.  Subcritical mass produces no outlier mass on compact subsets of $(1+\sqrt c,\infty)$, while the supercritical part is pushed forward by $\rho$ and weighted by \eqref{eq:MP-left-overlap} and \eqref{eq:MP-right-overlap}.
\end{example}

\begin{example}[Fixed rank as a special case]\label{ex:fixed-rank}
Nothing in the local transfer or projector theorem requires $r_n\to\infty$.  If $r_n=r$ is fixed, \cref{cor:isolated-vector} recovers the usual fixed-rank outlier and overlap limits.  Thus the theory is continuous across fixed, slowly growing, and mesoscopic rank; only the empirical normalization changes.
\end{example}

\begin{example}[A diverging background of stronger spikes]\label{ex:growing-strengths}
Suppose the subspace-law rate is
\[
  \eps_n\asymp\sqrt{\frac{r_n}{n}}+\sqrt{\frac{\log n}{n}},
\]
and consider a fixed superedge window containing a bounded group of spike strengths.  The remaining spikes need not be uniformly bounded: the local projector theorem remains valid provided
\[
  \norm{\Theta_n}_{\op}\eps_n\longrightarrow0.
\]
For example, if $r_n=n^\alpha$ with $0<\alpha<1$, the largest background strength may be $o(n^{(1-\alpha)/2})$, up to the logarithmic term.  The corresponding very large outliers lie outside the fixed window, but they do not invalidate the local conclusions for the bounded cluster.  This is the precise sense in which weighted compatibility is weaker than a blanket bound on every spike.
\end{example}

\begin{example}[Structured deterministic directions]\label{ex:structured-directions}
Under Gaussian or LSI independent-entry noise, $U_n$ and $V_n$ may consist of coordinate vectors, Fourier modes, wavelets, graph eigenvectors, or any other deterministic orthonormal frames.  No incoherence or delocalization of the signal directions is required for the first-order results, because the anisotropic local law and concentration estimates are uniform over deterministic unit vectors.
\end{example}

\begin{example}[Concrete non-Gaussian LSI noises]\label{ex:non-gaussian-lsi}
After centering and rescaling to variance one, \cref{thm:lsi-entry} applies to i.i.d.\ entries with density proportional to $e^{-V(x)}$ whenever $V''\ge\kappa>0$.  This includes, for example, densities proportional to $e^{-x^2/2-ax^4}$ with $a>0$.  It also includes bounded perturbations $e^{-V(x)-W(x)}$ when $V''\ge\kappa>0$ and $W$ has bounded oscillation.  Real and imaginary coordinates may be chosen independently from such families in the complex model.
\end{example}

\begin{example}[Rademacher and bounded-entry noise at logarithmic rank]\label{ex:rademacher-log-rank}
Rademacher entries and, more generally, centered bounded i.i.d.\ entries satisfy the standard overwhelming-probability anisotropic Marchenko--Pastur law.  Hence \cref{cor:poly-local-law} covers arbitrary deterministic signal frames of rank $O(\log n)$ with bounded strengths, or with the corresponding weighted compatibility, even though the continuous-gradient LSI argument in \cref{thm:lsi-entry} is not being invoked.
\end{example}

\begin{example}[Random strengths and random but non-Haar directions]\label{ex:random-signal-independent}
In the independent-entry model of \cref{thm:lsi-entry}, the entire signal $P_n$ may be random, including its strengths and directions, provided it is independent of $X_n$ and satisfies the rank and compatibility conditions with high probability.  Conditioning on $P_n$ reduces the proof to deterministic frames.  Thus Haar-distributed signal directions are not required in this model class.
\end{example}

\section{Matrix-valued deterministic equivalents}\label{sec:matrix-valued}

The scalar subspace law is exactly what produces a universal one-dimensional outlier map $\rho$.  Some variance-profile or correlated ensembles have a deterministic resolvent equivalent that is not scalar on the signal spaces.  Such models are not failures of the transfer method; they simply have a genuinely matrix-valued outlier equation.

\begin{theorem}[Matrix-valued transfer principle]\label{thm:matrix-valued-transfer}
Assume \cref{ass:global-noise}.  Let $\Gamma$ be a deterministic contour with $\overline{\Omega_\Gamma}\Subset\C\setminus[-b,b]$, and let $\mathfrak M_n(\zeta)$ be a deterministic analytic $2r_n\times2r_n$ matrix on a neighborhood of $\overline{\Omega_\Gamma}$.  Suppose
\begin{align*}
  \sup_{\zeta\in\Gamma}\norm{\calG_n(\zeta)-\mathfrak M_n(\zeta)}_{\op}&\to0,\\
  \sup_{\zeta\in\Gamma}
  \norm{(\calG_n(\zeta)-\mathfrak M_n(\zeta))\calK_n}_{\op}&\to0
\end{align*}
in probability.  Assume
\[
  \sup_{n,\zeta\in\Gamma}
  \norm{(\Id-\mathfrak M_n(\zeta)\calK_n)^{-1}}_{\op}<\infty.
\]
Then, with probability tending to one, the eigenvalues of $\calW_n$ enclosed by $\Gamma$ have the same total multiplicity as the zeros in $\Omega_\Gamma$ of
\[
  \det(\Id-\mathfrak M_n(\zeta)\calK_n).
\]
Writing
\[
  \widehat\calP_n(\Gamma):=\frac1{2\pi\ii}\int_\Gamma(\zeta\Id-\calW_n)^{-1}\,\dd\zeta
\]
for the corresponding Riesz projector, one has
\begin{equation}\label{eq:matrix-valued-projector}
 \calE_n^*\widehat\calP_n(\Gamma)\calE_n
 -\frac1{2\pi\ii}\int_\Gamma
   (\Id-\mathfrak M_n(\zeta)\calK_n)^{-1}\mathfrak M_n(\zeta)\,\dd\zeta
 \longrightarrow0
\end{equation}
in operator norm and in probability.
\end{theorem}

\begin{proof}
Repeat the homotopy and compressed-resolvent proof of \cref{prop:contour-transfer}, with $\calM$ replaced by $\mathfrak M_n$.
\end{proof}

\begin{remark}[Scope of the scalar theory]\label{rem:scope-scalar}
When $\mathfrak M_n$ is not asymptotically scalar on the signal subspaces, spike directions interact with the anisotropy of the noise and a single map $\theta\mapsto\rho(\theta)$ is generally impossible.  \Cref{thm:matrix-valued-transfer} still gives values and projectors, but they are determined by a matrix equation.  The scalar results of this paper therefore make a precise, transparent claim: they apply whenever noise anisotropy is averaged away by Haar signal/noise orientations or is absent at the level of the anisotropic local law.
\end{remark}

\appendix

\section{Stiefel concentration and uniform compression}\label{app:stiefel}

We record the concentration estimates used in \cref{sec:random-orientation,sec:haar-noise}.  The constants below are absolute and may change from line to line.

\begin{lemma}[Sphere quadratic forms]\label{lem:sphere-quadratic}
Let $w$ be uniform on the unit sphere of $\F^N$, and let $D$ be deterministic Hermitian with $\norm D_{\op}\le M$.  Then
\begin{equation}\label{eq:sphere-quadratic-tail}
 \Pp\left(\left|w^*Dw-\frac1N\tr D\right|>t\right)
 \le 2\exp\left(-cN\min\left(\frac{t^2}{M^2},\frac tM\right)\right)
\end{equation}
for all $t>0$.  In particular the Gaussian form $2\exp(-cNt^2/M^2)$ holds for $0<t\le M$, and on any range $t\le CM$ the two exponents agree up to the constant $c$.
\end{lemma}

\begin{proof}
Write $w=g/\norm g_2$ with $g$ standard real or complex Gaussian.  After diagonalizing $D$, the centered numerator is a weighted sum of independent centered subexponential variables $|g_j|^2-1$.  Bernstein's inequality and chi-square concentration for $\norm g_2^2$ give \eqref{eq:sphere-quadratic-tail}.  See, for example,~\cite{Vershynin2018}.
\end{proof}

\begin{lemma}[Independent-sphere bilinear forms]\label{lem:sphere-bilinear}
Let $w_1$ and $w_2$ be independent uniform unit vectors in $\F^{N_1}$ and $\F^{N_2}$, and let $D\in\F^{N_1\times N_2}$ satisfy $\norm D_{\op}\le M$.  Then
\begin{equation}\label{eq:sphere-bilinear-tail}
  \Pp(|w_1^*Dw_2|>t)
  \le2\exp\left(-c\min(N_1,N_2)\frac{t^2}{M^2}\right).
\end{equation}
\end{lemma}

\begin{proof}
Condition on $w_1$.  The vector $D^*w_1$ has norm at most $M$, and its inner product with an independent uniform sphere vector has a subgaussian tail at scale $M/\sqrt{N_2}$.  Interchange the two sides if necessary.
\end{proof}

\begin{lemma}[Net bounds]\label{lem:net-bound}
A $1/4$-net $\mathcal N$ of the unit sphere in $\F^r$ may be chosen with $|\mathcal N|\le9^{2r}$.  For Hermitian $A$ and arbitrary $B$,
\begin{align*}
  \norm A_{\op}&\le2\max_{x\in\mathcal N}|x^*Ax|,\\
  \norm B_{\op}&\le2\max_{x,y\in\mathcal N}|x^*By|.
\end{align*}
\end{lemma}

\begin{lemma}[Haar Stiefel compression]\label{lem:stiefel-compression}
Let $Q\in\St(N,r)$ be Haar and let $D\in\F^{N\times N}$ be deterministic with $\norm D_{\op}\le M$.  Then, for every $u\ge0$,
\begin{equation}\label{eq:stiefel-compression-tail}
 \Pp\left(
   \norm{Q^*DQ-\frac1N(\tr D)\Id_r}_{\op}
   >CM\sqrt{\frac rN}+u
 \right)
 \le C\exp\left(-cN\frac{u^2}{M^2}\right).
\end{equation}
The same conclusion holds for non-Hermitian $D$ over $\F=\C$ after changing the constants.  In the real models used here with complex spectral parameter, it also applies to complex symmetric $D$ by treating its real and imaginary symmetric parts separately.
\end{lemma}

\begin{proof}
For Hermitian $D$, fix a $1/4$-net $\mathcal N$.  For each $x\in\mathcal N$, $Qx$ is uniform on the unit sphere, so \cref{lem:sphere-quadratic} applies.  A union bound over $|\mathcal N|$ and \cref{lem:net-bound} give the stated threshold for the nontrivial range of $u$; after increasing $C$, the event is empty for larger $u$ because both compression terms have norm at most $M$.  Over $\F=\C$, decompose a general $D$ into its Hermitian and skew-Hermitian parts.  In the real complex-symmetric case, apply the real Hermitian estimate separately to $\operatorname{Re}D$ and $\operatorname{Im}D$.
\end{proof}

\begin{lemma}[Mixed Haar Stiefel compression]\label{lem:stiefel-mixed}
Let $Q_1\in\St(N_1,r)$ and $Q_2\in\St(N_2,r)$ be independent Haar matrices, and let $D\in\F^{N_1\times N_2}$ satisfy $\norm D_{\op}\le M$.  Then
\begin{equation}\label{eq:stiefel-mixed-tail}
 \Pp\left(
   \norm{Q_1^*DQ_2}_{\op}
   >CM\sqrt{\frac r{\min(N_1,N_2)}}+u
 \right)
 \le C\exp\left(-c\min(N_1,N_2)\frac{u^2}{M^2}\right).
\end{equation}
In the real model with a complex matrix $D$, the same conclusion holds after changing constants, by applying the estimate separately to the real and imaginary parts.
\end{lemma}

\begin{proof}
For fixed net points $x,y$, the vectors $Q_1x$ and $Q_2y$ are independent uniform sphere vectors.  Apply \cref{lem:sphere-bilinear}, take a union bound over the net pairs, and use \cref{lem:net-bound}.  For complex $D$ in the real model, write $D=\operatorname{Re}D+\ii\operatorname{Im}D$ and apply the real estimate to the two parts.
\end{proof}

\begin{lemma}[Uniformity in the spectral parameter]\label{lem:z-uniformity}
Let $K\subset\C$ be compact, let $1\le r\le N$, let $M>0$, and let $Z(\zeta)$ be a nonnegative random process whose sample paths are $L$-Lipschitz on $K$.  Assume that $L/M\le N^{A_0}$ for some fixed $A_0<\infty$ (in the applications below, $L$ and $M$ are bounded on $K$).  Suppose that for every fixed $\zeta\in K$ and every $u\ge0$,
\[
  \Pp\left(Z(\zeta)>A M\sqrt{\frac rN}+u\right)
  \le C\exp\left(-cN\frac{u^2}{M^2}\right),
\]
where $A,C,c$ are independent of $N,r,\zeta$.  Then
\[
  \sup_{\zeta\in K}Z(\zeta)
  =O_{\Pp}\left(
      M\sqrt{\frac rN}
      +M\sqrt{\frac{\log N}{N}}
    \right).
\]
In applications, $Z(\zeta)$ is a centered diagonal or mixed compression error; the pathwise Lipschitz bound follows from the operator-norm Lipschitz continuity of the underlying matrix family.
\end{lemma}

\begin{proof}
If $L=0$, there is nothing to prove.  Otherwise, use a two-dimensional mesh of spacing proportional to $M\sqrt{\log N/N}/L$.  The polynomial bound on $L/M$ makes its cardinality polynomial in $N$.  Taking $u=C_0M\sqrt{\log N/N}$ with $C_0$ large, a union bound controls all mesh points with probability tending to one, and the pathwise Lipschitz bound fills the gaps.
\end{proof}

\section{Auxiliary analytic facts}\label{app:analytic}

\begin{lemma}[Monotonicity of the rectangular transform]\label{lem:d-monotone}
For $z>b$,
\[
  -\phi'(z)\ge z^{-2},\qquad
  -\widetilde\phi'(z)\ge z^{-2},\qquad
  -d'(z)\ge2z^{-3}.
\]
In particular, $d$ is strictly decreasing.
\end{lemma}

\begin{proof}
For $x\ge0$ and $z>\sqrt x$,
\[
  \frac{\partial}{\partial z}\frac{z}{z^2-x}
  =-\frac{z^2+x}{(z^2-x)^2}\le-z^{-2}.
\]
Integrate this inequality against $\mu$ to get $-\phi'(z)\ge z^{-2}$.  For the second bound, differentiate $\widetilde\phi=c\phi+(1-c)/\zeta$ to get $-\widetilde\phi'(z)=c(-\phi'(z))+(1-c)z^{-2}\ge z^{-2}$, using $c\le1$.  Since $\phi(z),\widetilde\phi(z)\ge z^{-1}$, the product rule gives the last bound.
\end{proof}

\begin{lemma}[Continuity and boundedness of overlap weights]\label{lem:weight-continuity}
The functions $\cL,\cR,\cLR$ are continuous on every compact subset of $(\theta_\star,\infty)$ and are bounded there.  All three take values in $[0,1]$.
\end{lemma}

\begin{proof}
Continuity and boundedness follow directly from \eqref{eq:overlap-weights}, the smoothness of $d$ on $(b,\infty)$, and \cref{lem:d-monotone}.  We prove the range directly.  Fix $z>b$ and set
\[
 m_0:=\int\frac1{z^2-x}\,\mu(\dd x),\qquad
 m_1:=\int\frac1{(z^2-x)^2}\,\mu(\dd x),
\]
\[
 q:=z^2m_0=z\phi(z),\qquad s:=z^4m_1,
 \qquad r:=z\widetilde\phi(z)=1-c+cq.
\]
Then $q\ge1$ and, by Cauchy--Schwarz, $s\ge q^2$.  Since
\[
 -z^2\phi'(z)=2s-q,
 \qquad
 \widetilde\phi'(z)=c\phi'(z)-\frac{1-c}{z^2},
\]
a direct calculation gives
\[
 D:=-z^3d'(z)
   =(2s-q)(1-c+2cq)+(1-c)q>0.
\]
At a supercritical root $z=\rho(\theta)$, the relation $d(z)=\theta^{-2}$ rewrites the two diagonal weights as
\[
 \cL(\theta)=\frac{2q^2r}{D},
 \qquad
 \cR(\theta)=\frac{2qr^2}{D}.
\]
Using $s\ge q^2$ and $q\ge1$,
\begin{align*}
 D-2q^2r
 &\ge 2cq^2(q-1)\ge0,\\
 D-2qr^2
 &\ge 2q(q-1)\bigl[1+c(2-c)(q-1)\bigr]\ge0.
\end{align*}
Thus $0\le\cL,\cR\le1$.  Finally, \eqref{eq:cross-weight} gives
$0\le\cLR=\sqrt{\cL\cR}\le1$.
\end{proof}

\begingroup
\renewcommand{\addcontentsline}[3]{}

\section*{Funding}
The research presented in this paper was supported by the European Research Council (ERC) under the European Union's Horizon Europe research and innovation program (grant agreement No.~101041711), by the Simons Foundation through the Collaboration on the Mathematical and Scientific Foundations of Deep Learning, by Heights Labs, by Convex Nexus Capital, and by the Israel Science Foundation (grants 2258/19 and 4101/25).

\section*{Conflicts of interest}
Y.~Shmalo holds equity interests in Heights Labs and Convex Nexus Capital.

\section*{Declaration of generative AI use}
Generative-AI tools were used for language editing, organizational assistance, and checking algebraic presentation during manuscript preparation.  The author independently reviewed the mathematical statements, proofs, and references and takes full responsibility for the content of the paper.

\endgroup

\end{document}